\documentclass{article}

\usepackage[margin=1.2in]{geometry}
\usepackage{bm}
\usepackage{afterpage}
\usepackage{mathrsfs}
\usepackage{array}
\usepackage{booktabs}
\usepackage{makecell}
\usepackage{diagbox}
\usepackage{multirow}
\usepackage{amsmath}
\numberwithin{equation}{section}
\usepackage{amsthm}
\usepackage{amssymb}
\usepackage[hidelinks]{hyperref}
\usepackage{cleveref}
\crefformat{equation}{(#2#1#3)}
\Crefformat{equation}{(#2#1#3)}
\usepackage{graphicx}
\usepackage{algorithm}
\usepackage{algorithmic}
\usepackage{lipsum}
\usepackage{amsfonts}
\usepackage{epstopdf}
\usepackage{amsopn}

\title{Frozen Gaussian Grid-point Correction For Semi-classical Schr{\"o}dinger Equation}

\author{Lihui Chai\thanks{School of Mathematics, Sun Yat-sen University, Guangzhou, China.}\and Zili Deng\thanks{School of Mathematics, Sun Yat-sen University, Guangzhou, China.}}

\begin{document}

\maketitle

\begin{abstract}
    We propose an efficient reconstruction algorithm named the frozen Gaussian grid-point correction (FGGC) for computing the Schr{\"o}dinger equation in the semi-classical regime using the frozen Gaussian approximation (FGA). The FGA has demonstrated its superior efficiency in dealing with semi-classical problems and high-frequency wave propagations. However, reconstructing the wave function from a large number of Gaussian wave-packets is typically computationally intensive. This difficulty arises because these wave-packets propagate along the FGA trajectories to non-grid positions, making the application of the fast Fourier transform infeasible. In this work, we introduce the concept of ``on-grid correction'' and derive the formulas for the least squares approximation of Gaussian wave-packets, and also provide a detailed process of the FGGC algorithm. Furthermore, we rigorously prove that the error introduced by the least squares approximation on each Gaussian wave-packet is independent of the semi-classical parameter $\varepsilon$. Numerical experiments show that the FGGC algorithm can significantly improve reconstruction efficiency while introducing only negligible error, making it a powerful tool for solving the semi-classical Schr{\"o}dinger equation,  especially in applications requiring both accuracy and efficiency.
\end{abstract}


\section{Introduction}\label{sec:1}
The Schr{\"o}dinger equation lies at the heart of quantum mechanics, providing the fundamental framework for describing the evolution of quantum systems over time \cite{GriffithsQuantumMechanics}. Research on the Schr{\"o}dinger equation is crucial for uncovering fundamental properties of matter and advancing fields such as materials science, condensed matter physics, and quantum information science \cite{electronic_structures}. In this paper, we consider numerical methods for the semi-classical Schr{\"o}dinger equation:
\begin{equation}
    \mathrm{i}\varepsilon\frac{\partial u}{\partial t} = -\frac{\varepsilon^2}{2}\Delta u + V(\bm{x})u,\quad (t,\bm{x})\in\mathbb{R}\times\mathbb{R}^d,
    \label{eq:1.1}
\end{equation}
with the WKB initial condition:
\begin{equation}
    u_0(\bm{x})=\sqrt{n_0(\bm{x})}\exp\left(\frac{\mathrm{i}}{\varepsilon}S_0(\bm{x})\right).
    \label{eq:1.2}
\end{equation}
Here, $\varepsilon$ is the semi-classical parameter describing the ratio of microscopic to macroscopic scales, with $0<\varepsilon\ll1$ in the semi-classical regime; $d$ is the spatial dimension; $\mathrm{i}=\sqrt{-1}$ is the imaginary unit; $u(t,\bm{x})$ is the complex-valued wave function; $V(\bm{x})$ is a given real-valued potential; $\Delta=\Delta_{\bm{x}}$ is the Laplace operator; $n_0(\bm{x})$ and $S_0(\bm{x})$ are smooth, real-valued, and independent of $\varepsilon$; and $n_0(\bm{x})$ decays sufficiently fast as $\left|\bm{x}\right|\to\infty$.

The properties of the Schr{\"o}dinger equation have been widely studied from both theoretical and numerical perspectives. It is well known that the wave function $u(t,\bm{x})$ exhibits oscillations of wavelength $\varepsilon$ in space and time and does not converge strongly as $\varepsilon\to0^+$. However, the weak convergence of the wave function is insufficient for passing to the limit in macroscopic densities, making theoretical analysis of the semi-classical limit mathematically challenging. Since the 1990s, significant progress has been made in the study of the semi-classical limit of the Schr{\"o}dinger equation. Using microlocal analysis tools such as H-measures \cite{H-measures} and Wigner measures \cite{Wigner-measures_Lions,Wigner_transform,WignerSP,WignerFigalli}, researchers have rigorously shown that the semi-classical limit of the Schr{\"o}dinger equation for a range of complex physical systems is described by classical dynamical equations in phase space. These results illuminate the connection and transition between quantum mechanics and classical mechanics. 

Meanwhile, the numerical solution of the Schr{\"o}dinger equation requires substantial computational resources to achieve accurate physical observables when $\varepsilon$ is small. The challenge becomes even more pronounced when accurate wave functions are of interest \cite{actanumerica}. Existing numerical algorithms for solving the Schr{\"o}dinger equation can be broadly classified into two categories: direct discretization methods and asymptotic methods. Direct discretization methods, such as the finite difference method \cite{Schro_FD1,Schro_FD2} and the time-splitting spectral method (TSSP) \cite{tssp}, have been developed over several decades and form a robust class of algorithms with high accuracy. However, these methods often suffer from low efficiency and high mesh cost: the time-splitting spectral method requires $\Delta t=O(\varepsilon)$ and $\Delta x=O(\varepsilon)$ to guarantee an accurate approximation of the wave function, while the finite-difference method requires even $\Delta t=o(\varepsilon)$ and $\Delta x=o(\varepsilon)$. On the other hand, asymptotic methods such as the WKB approximation \cite{WKB}, the Gaussian beam method \cite{GBM}, and the frozen Gaussian approximation (FGA) \cite{FGAFracSchro,FGAwave} may require relatively fewer computational resources in the semi-classical regime. However, the WKB solution becomes unreliable after the formation of caustics, and the Gaussian beam method also loses accuracy as the beam evolves and propagates, causing its width to increase. 

The frozen Gaussian approximation was initially proposed in quantum chemistry to simulate the Schr{\"o}dinger equation \cite{FGASchro1981,FGASchro1994,FGASchro2006,FGASchro2007,FGASchro2016}, using fixed-width Gaussian functions to approximate the wave function. Recently, FGA theory has been successfully extended and applied to a variety of systems and equations, such as linear strictly hyperbolic systems \cite{fgaconv}, non-strictly hyperbolic systems including elastic wave equations \cite{FGAelastic} and relativistic Dirac equations in the semi-classical regime \cite{FGAnonstrictlyhyperbolic,FGAdirac}, non-adiabatic dynamics in surface hopping problems \cite{FGAsh1,FGAsh2,FGAsh3}, and open quantum systems \cite{open_quantum_system}. These investigations demonstrate the adaptability and effectiveness of FGAs in solving intricate problems in quantum mechanics and high-frequency wave phenomena.

In this paper, we aim to design an efficient reconstruction algorithm for the FGA. Practical experience shows that much more computational effort is required for the wave reconstruction step compared to the initial decomposition step. The initial decomposition can be performed on grid points and utilizes the fast Fourier transform (FFT) to accelerate the calculation of the numerical integral. However, this approach is not directly applicable to the wave reconstruction step, as the grid points evolve over time into off-grid points, complicating the reconstruction process. To enhance reconstruction efficiency, we propose decomposing off-grid Gaussian wave-packets into several on-grid Gaussian wave-packets, which can be efficiently implemented using the least squares approximation. This algorithm is named the frozen Gaussian grid-point correction (FGGC). We then conduct error analysis and numerical experiments, rigorously proving that the error of decomposition for a single Gaussian wave-packet is independent of $\varepsilon$, and numerically demonstrating that the overall error of the FGGC algorithm is also independent of $\varepsilon$. A multi-step solver is also proposed to further improve efficiency for long-time simulations. The results show that, using a basic ``on-grid correction'' idea, the FGGC algorithm can improve the efficiency of the reconstruction step by an order of magnitude, with only a small and negligible error cost.

The remainder of this paper is organized as follows: In \Cref{sec:2}, we begin with a brief review of the FGA formulation and then consider the least squares approximation (LSA) for Gaussian wave-packets. Building upon this LSA problem, we present the basic procedure of the frozen Gaussian grid-point correction algorithm. Subsequently, in \Cref{sec:3}, we perform an error estimate for the FGGC algorithm. Additionally, we propose improved versions of the FGGC that further reduce computational cost. A series of numerical experiments is conducted in \Cref{sec:4}, including cases in one to three dimensions. These results demonstrate that the FGGC algorithm significantly improves computational efficiency during wave reconstruction while introducing only a minor error. Moreover, the FGGC algorithm can further enhance computing performance for long-time evolution problems by employing a multi-step version solver. \Cref{sec:5} concludes this paper.

\section{Frozen Gaussian Approximation and Frozen Gaussian grid-point correction}\label{sec:2}

In this section, we briefly recap the frozen Gaussian approximation to the semi-classical Schr{\"o}dinger equation and introduce the FGGC algorithm employing least squares approximations for wave-packets.

\subsection{The frozen Gaussian approximation for semi-classical Schr{\"o}dinger equation}\label{sec:2.1}

The frozen Gaussian approximation is an asymptotic approximation to \Cref{eq:1.1}, which gives an integral representation of the solution of the wave function:
\begin{equation}
    u_\text{FGA}(t,\bm{x})=\frac{2^{\frac{d}{2}}}{(2\pi\varepsilon)^{\frac{3d}{2}}}\int_{\mathbb{R}^{2d}}A(t,\bm{q},\bm{p})e^{\frac{\mathrm{i}}{\varepsilon}\Phi(t,\bm{x},\bm{q},\bm{p})}\ \mathrm{d}\bm{p}\ \mathrm{d}\bm{q},
    \label{eq:2.1}
\end{equation}
where the phase function is defined as
\begin{equation}
    \Phi(t,\bm{x},\bm{q},\bm{p})=S(t,\bm{q},\bm{p})+\bm{P}(t,\bm{q},\bm{p})\cdot(\bm{x}-\bm{Q}(t,\bm{q},\bm{p}))+\frac{\mathrm{i}}{2}|\bm{x}-\bm{Q}(t,\bm{q},\bm{p})|^2.
    \label{eq:2.2}
\end{equation}

In this integral representation, the position center $\bm{Q}$, momentum center $\bm{P}$, amplitude $A$ and action $S$ are time-dependent variables satisfying an ODE system:
\begin{alignat}{4}
    \frac{\mathrm{d}\bm{Q}}{\mathrm{d}t}&= \bm{P}, &&\bm{Q}(0,\bm{q},\bm{p})=\bm{q},\label{eq:2.3} \\
    \frac{\mathrm{d}\bm{P}}{\mathrm{d}t}&= -\nabla V(\bm{Q}), &&\bm{P}(0,\bm{q},\bm{p})=\bm{p},\label{eq:2.4} \\
    \frac{\mathrm{d}S}{\mathrm{d}t}&= \frac{1}{2}|\bm{P}|^2 - V(\bm{Q}), &&S(0,\bm{q},\bm{p})=0,\label{eq:2.5}\\
    \frac{\mathrm{d}A}{\mathrm{d}t}&=\frac{A}{2}\operatorname{tr}\left(\bm{Z}^{-1}\cdot\frac{\mathrm{d}\bm{Z}}{\mathrm{d}t}\right),\ &&A(0,\bm{q},\bm{p})=2^{\frac{d}{2}}\widehat{u_0}(\bm{q},\bm{p}),\label{eq:2.6}
\end{alignat}
where
\begin{equation}
    \widehat{u}_0(\bm{q},\bm{p})=\int_{\mathbb{R}^d}u_0(\bm{y})e^{\frac{\mathrm{i}}{\varepsilon}\big(-\bm{p}\cdot(\bm{y}-\bm{q})+\frac{\mathrm{i}}{2}|\bm{y}-\bm{q}|^2\big)}\mathrm{d}\bm{y}.\label{eq:initial_decomposition}
\end{equation}
In the above ODE system, we have used the shorthand notation:
\begin{equation}\begin{split}
    \partial_{\bm{z}}=\partial_{\bm{q}}-\mathrm{i}\partial_{\bm{p}}\ \ \ \operatorname{and} \ \ \ \bm{Z}=\partial_{\bm{z}}(\bm{Q}+\mathrm{i}\bm{P}),
    \label{eq:matrixZ_notation}
\end{split}\end{equation}
where $\partial_{\bm{z}}\bm{Q}$,$\partial_{\bm{z}}\bm{P}$ and $\bm{Z}$ are understood as matrices. A detailed derivation for formulas \Cref{eq:2.3}$\sim$\Cref{eq:initial_decomposition} can be found in \cite{FGAwave}. The following proposition shows that the FGA ansatz given by \Cref{eq:2.1} is exact for the initial condition \Cref{eq:initial_decomposition}.

\newtheorem{proposition}{Proposition}
\begin{proposition}\label{thm:a_initial_condition}
    Let $u_0(\bm{x})\in L^2(\mathbb{R}^{d})$ be the initial condition of the Schr{\"o}dinger equation, then we have 
    \begin{equation}\begin{aligned}
        u_0(\bm{x})=\frac{1}{(2\pi\varepsilon)^{\frac{3d}{2}}}\int_{\mathbb{R}^{2d}}A(0,\bm{q},\bm{p})e^{\frac{\mathrm{i}}{\varepsilon}\Phi(0,\bm{x},\bm{q},\bm{p})}\mathrm{d}\bm{q}\ \mathrm{d}\bm{p},
        \label{eq:2.9}
    \end{aligned}\end{equation}
    where the initial condition is given by \Cref{eq:2.3}$\sim$\Cref{eq:initial_decomposition}.
\end{proposition}

The FGA approximates the Schr{\"o}dinger equation by three steps: First, decompose the initial waves into several Gaussian wave-packets in the phase space to determine $A(0,\bm{q},\bm{p})$ by \Cref{eq:initial_decomposition} for each $(\bm{q},\bm{p})$. Second, propagate the wave-packets along the Hamiltonian flow by \Cref{eq:2.3}$\sim$\Cref{eq:2.6}. Third, reconstruct the wave function using the wave-packets after propagation. 

In practical algorithms, the FGA method requires the calculation of a series of discrete summations to approximate the integrals in initial decomposition and wave reconstruction. For initial decomposition, FFT can be utilized to accelerate the calculation. However, for the wave reconstruction, the process can be approximated by the following discrete form
\begin{equation}\begin{aligned}
    u(t,\bm{x})&\approx\ \frac{(\Delta q\Delta p)^d}{(2\pi\varepsilon)^{\frac{3d}{2}}}\sum_{j=1}^{M}A_je^{\frac{\mathrm{i}}{\varepsilon}\left(S_j+\bm{P}_j\cdot\left(\bm{x}-\bm{Q}_j\right)\right)-\frac{1}{2\varepsilon}\left(\bm{x}-\bm{Q}_j\right)^2},
    \label{eq:2.10}
\end{aligned}\end{equation}
in which FFT is no longer feasible. This is because the wave-packets generally do not align with the grid points of the phase space. Consequently, the computational time for wave reconstruction significantly exceeds that of the initial decomposition, and this computational time escalates rapidly with an increase in dimensionality $d$ or a decrease in the semi-classical parameter $\varepsilon$.

\subsection{Least squares approximation problem and formulas for Gaussian wave-packets}\label{sec:2.2}

As mentioned earlier, summing up Gaussian wave-packets centered at off-grid points demands a significant amount of time. To address this issue, we propose a method for decomposing an off-grid Gaussian wave-packet into several on-grid wave-packets. A natural idea is to utilize the least squares approximation (LSA), which offers an ``optimal'' approximation in the $L^2$ sense. 

Generally, consider the least squares approximation of a complex-valued function $\phi(x)$ in a function space $V$ spanned by a set of basis functions $\left\{\psi^{(k)}\right\}_{k=1}^{n}$. We can express the approximation by 
\begin{equation}\begin{split}
    \phi^* \approx \sum_{k=1}^n c^{(k)} \psi^{(k)}.
    \label{eq:wave_packet_approximation}
\end{split}\end{equation}
The coefficients are determined by the normal equations:
\begin{equation}\begin{split}
    \sum_{k=1}^n A_{jk}c^{(k)}=f_j,\ \ j=1,2,\dots,n,
    \label{eq:normal_equation}
\end{split}\end{equation}
where
\begin{align}
    A_{jk}=\left<\psi^{(j)},\psi^{(k)}\right>\text{ and }f_j=\left<\psi^{(j)},\phi\right>,
    \label{eq:normal_equation_Af}
\end{align}
and the inner product $\langle\cdot,\cdot\rangle$ is defined for complex-valued functions in $L^2(\mathbb{R}^d)$ by 
\begin{equation}\begin{split}
    \langle g,h\rangle=\int_{\mathbb{R}^d}g(\bm{x})h(\bm{x})^\dagger\ \mathrm{d}\bm{x}.
    \label{eq:2.14}
\end{split}\end{equation}
In the FGA ansatz, the given function and the approximating functions are in the form of fixed-width Gaussian wave-packet. For instance:
\begin{equation}\begin{split}
    \phi(\bm{x})=G\left(\bm{x};\bm{Q}_0,\bm{P}_0,\varepsilon\right)\text{ and }
    \psi^{(k)}(\bm{x})=G\left(\bm{x};\bm{Q}_0^{(k)},\bm{P}_0^{(k)},\varepsilon\right),
    \label{eq:phi_and _psi_k}
\end{split}\end{equation}
in which $\bm{Q}_0^{(k)}$ and $\bm{P}_0^{(k)}$ should be the grid points near $\bm{Q}_0$ and $\bm{P}_0$ respectively. Here, $G(\bm{x};\bm{Q},\bm{P},\varepsilon)$ denotes the Gaussian wave-packet
\begin{equation}\begin{aligned}
    G(\bm{x};\bm{Q},\bm{P},\varepsilon)=e^{\frac{\mathrm{i}}{\varepsilon}\bm{P}\cdot(\bm{x}-\bm{Q})-\frac{1}{2\varepsilon}|\bm{x}-\bm{Q}|^2}.
    \label{eq:gaussian_wave_packet}
\end{aligned}\end{equation}
For simplicity, we may omit $\varepsilon$ and write $G(\bm{x};\bm{Q},\bm{P},\varepsilon)$ as $G(\bm{x};\bm{Q},\bm{P})$ when $\varepsilon$ is fixed and not of concern. 

Essentially, these inner products are interactions between two Gaussian wave-packets and are straightforward to compute. For instance, if $g(\bm{x})=G(\bm{x};\bar{\bm{Q}},\bar{\bm{P}})$ and $h(\bm{x})=G(\bm{x};\widetilde{\bm{Q}},\widetilde{\bm{P}})$, then
\begin{equation}\begin{split}
    \left<g,h\right>&=\int_{\mathbb{R}^d}g(\bm{x}) \, h(\bm{x})^\dagger\ \mathrm{d}\bm{x}\\
    &=(\pi\varepsilon)^{\frac{d}{2}}e^{-\frac{1}{4\varepsilon}\left|\widetilde{\bm{Q}}-\bar{\bm{Q}}\right|^2-\frac{1}{4\varepsilon}\left|\widetilde{\bm{P}}-\bar{\bm{P}}\right|^2+\frac{\mathrm{i}}{2\varepsilon}(\widetilde{\bm{Q}}-\bar{\bm{Q}})\cdot(\widetilde{\bm{P}}+\bar{\bm{P}})}.
    \label{eq:inner_product}
\end{split}\end{equation}
So, substitute $\widetilde{\bm{Q}},\bar{\bm{Q}},\widetilde{\bm{P}},\bar{\bm{P}}$ with $\bm{Q}^{(j)},\bm{Q}^{(k)},\bm{P}^{(j)},\bm{P}^{(k)}$, one can calculate all the entries of the coefficient matrix $\bm{A}$ and right-hand side vector $\bm{f}$ of equation \Cref{eq:normal_equation}. For example, for $d=1$, two nearby points can be chosen for each component as follows:
\begin{figure}
    \centering
    \includegraphics[width=0.6\linewidth]{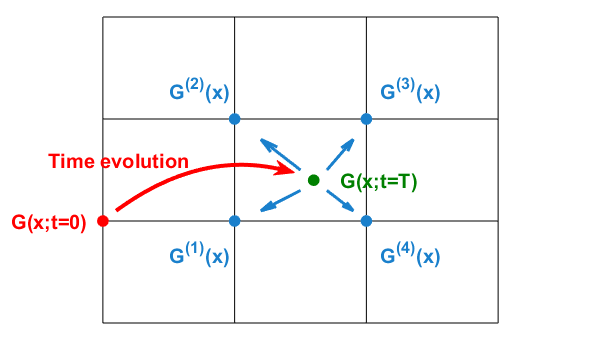}
    \caption{An on-grid wave-packet(red) evolves to an off-grid position(green) and then split into four on-grid wave-packets(blue).}
    \label{fig:pointsplit}
\end{figure}

\begin{equation}\begin{aligned}
     Q^{(1)} &= \left(\left[\frac{Q}{\Delta q}\right]\right)\Delta q,   &&P^{(1)} = \left(\left[\frac{P}{\Delta p}\right]\right)\Delta p,\\
     Q^{(2)} &= \left(\left[\frac{Q}{\Delta q}\right]+1\right)\Delta q, &&P^{(2)} = \left(\left[\frac{P}{\Delta p}\right]\right)\Delta p,\\
     Q^{(3)} &= \left(\left[\frac{Q}{\Delta q}\right]\right)\Delta q,   &&P^{(3)} = \left(\left[\frac{P}{\Delta p}\right]+1\right)\Delta p,  \\
     Q^{(4)} &= \left(\left[\frac{Q}{\Delta q}\right]+1\right)\Delta q ,&&P^{(4)} = \left(\left[\frac{P}{\Delta p}\right]+1\right)\Delta p.
    \label{eq:neighbor_example}
\end{aligned}\end{equation}
Here, $\left[\ \cdot\ \right]$ denotes the floor function. As illustrated in \Cref{fig:pointsplit}, an on-grid Gaussian wave-packet evolves to an off-grid position along the FGA trajectory, and is then decomposed into four neighboring on-grid Gaussian wave-packets via least squares approximation. In particular, for $\varepsilon=1/64$, $\Delta p=\pi/32$, $\Delta q=1/16$, and $0\le Q<\Delta q$, $0\le P<\Delta p$, then \Cref{eq:neighbor_example} yields $Q^{(k)}=0$ or $\Delta q$ and $P^{(k)}=0$ or $\Delta p$. Each term in \Cref{eq:normal_equation} can be computed using \Cref{eq:inner_product}, and then the normal equations can be solved to obtain the coefficients. The coefficients and errors for the least squares approximation as functions with respect to $Q$ and $P$, are shown in \Cref{fig:LSA_coef_err}. The error of the least squares approximation using four nearby points is acceptable, and numerical experiments in \Cref{sec:3.1} show that the error can be further reduced by considering more neighbors.

\begin{figure}[t]
    \centering
    \includegraphics[width=1\linewidth]{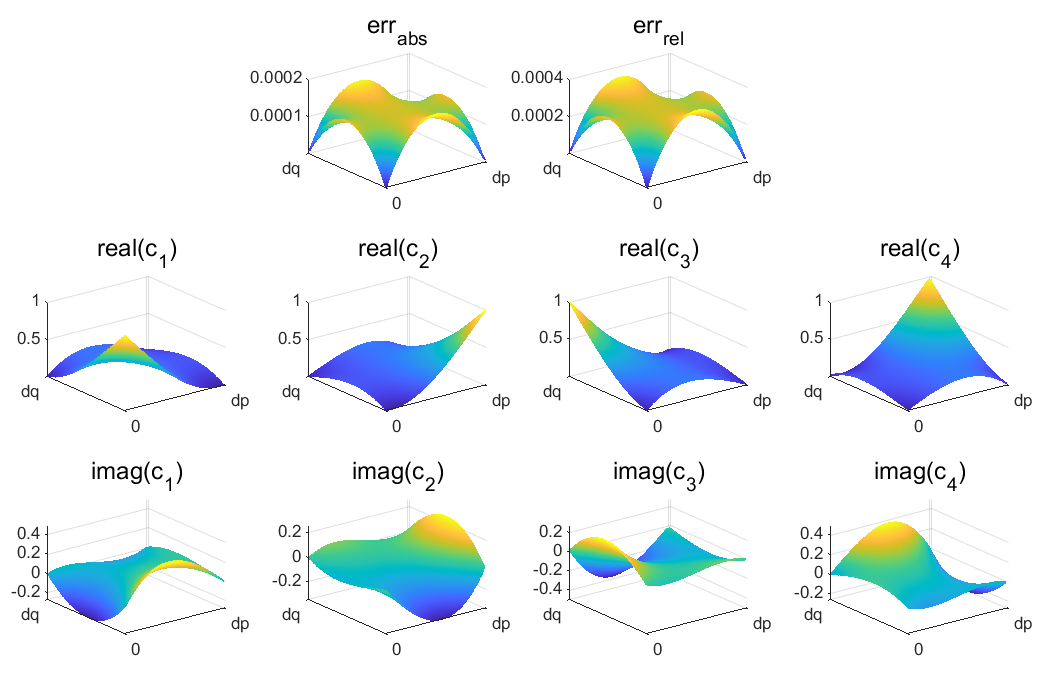}
    \caption{The LSA error and coefficients, in the case of $d=1$, neighbor strategy \textbf{Q2P2}, $\varepsilon=1/64$, $\Delta q=1/16$, $\Delta p=\pi/32$, and $(Q_0,P_0)\in[0,\Delta q)\times[0,\Delta p)$.}
    \label{fig:LSA_coef_err}
\end{figure}

\subsection{The frozen Gaussian grid-point correction algorithm}\label{sec:2.3}

Suppose that for any $(\bm{q},\bm{p})$ in the phase space, $A(0,\bm{q},\bm{p})$ has already been calculated by \Cref{eq:initial_decomposition}. By evolving the ODE system \Cref{eq:2.3} $\sim$ \Cref{eq:initial_decomposition} up to time $t$, one can reconstruct the wave function as 
\begin{equation}\begin{aligned}
    u(t,\bm{x})&\approx\ \frac{(\Delta q\Delta p)^d}{(2\pi\varepsilon)^{\frac{3d}{2}}}\sum_{j=1}^{M}A_j e^{\frac{\mathrm{i}}{\varepsilon}S_j}G\left(\bm{x};\bm{Q}_j,\bm{P}_j\right).
    \label{eq:summing_wave_packet}
\end{aligned}\end{equation}
Here, $\Delta q$ and $\Delta p$ denote the discrete step size for numerical integration, $M$ denotes the number of wave-packets included in the calculation. For simplicity, we omit the explicit time dependence such as $A_j=A_j(t)$ and adopt the notation $G_j(\bm{x})=G\left(\bm{x};\bm{Q}_j,\bm{P}_j\right)$ in the following text. For any $j$, using the formulas introduced previously, we can perform the least squares approximation for each $G_j(\bm{x})$. We first introduce some notations and definitions for choosing a set of neighbors. 
\newtheorem{definition}{Definition}
\begin{definition}
    Assume that the semi-classical parameter $\varepsilon$ and the discrete mesh step $\Delta q$ and $\Delta p$ are given. Then, choose $2n$ sequences of shift integers $\bm{\delta q}^{(k)}=\left(\delta q^{(k)}_l\right)_{l=1}^d$ and  $\bm{\delta p}^{(k)}=\left(\delta p^{(k)}_l\right)_{l=1}^d$ for $k=1,2,\cdots,n$, and define the neighbor points map $\bm{\mathcal{N}}(\bm{Q},\bm{P};k)=\left(\bm{Q}^{(k)},\bm{P}^{(k)}\right)$ by:
    \begin{equation}\begin{aligned}
        \bm{Q}^{(k)}=&\operatorname{floor}\left(\frac{\bm{Q}}{\Delta q}\right)\Delta q+\bm{\delta q}^{(k)}\Delta q,\\
        \bm{P}^{(k)}=&\operatorname{floor}\left(\frac{\bm{P}}{\Delta p}\right)\Delta p+\bm{\delta p}^{(k)}\Delta p,
        \label{eq:neighbor_representation}
    \end{aligned}\end{equation}
    where $\operatorname{floor}(\cdot)$ denotes the floor function.
\end{definition}

Therefore, the neighbor point map \Cref{eq:neighbor_representation} can be interpreted as a ``neighbor strategy'': selecting a set of $\delta q_l^{(k)}$ and $\delta p_l^{(k)}$ generates $n$ distinct on-grid neighbor wave-packets for the target wave-packet. Now we choose a set of neighbors for each specific wave-packet $G_j$, where the $k$-th neighbor can be represented as
\begin{equation}\begin{aligned}
    G_j^{(k)}(\bm{x}):=G(\bm{x};\bm{\mathcal{N}}(\bm{Q}_j,\bm{P}_j;k)).
    \label{eq:k_th_neighbor}
\end{aligned}\end{equation}
Then the least squares approximation
\begin{equation}\begin{aligned}
    G_j(\bm{x})\approx\sum_{k=1}^{n}c_j^{(k)}G_j^{(k)}(\bm{x}),
    \label{eq:2.22}
\end{aligned}\end{equation}
can be obtained by solving the normal equations \Cref{eq:normal_equation} with $\phi=G_j$ and $\psi^{(k)}=G_j^{(k)}$. Substituting \Cref{eq:2.22} into \Cref{eq:summing_wave_packet} yields
\begin{equation}\begin{aligned}
    u(t,\bm{x})\approx\ \frac{(\Delta q\Delta p)^d}{(2\pi\varepsilon)^{\frac{3d}{2}}}\sum_{j=1}^{M}A_j e^{\frac{\mathrm{i}}{\varepsilon}S_j}\sum_{k=1}^{n}c_j^{(k)}G_j^{(k)}(\bm{x}).
    \label{eq:2.23}
\end{aligned}\end{equation}
Note that the parameters of $G_j^{(k)}$, namely $\bm{Q}^{(k)}_j$ and $\bm{P}^{(k)}_j$, are all on-grid. Thus, we can combine Gaussian wave-packets sharing the same $\bm{Q}_j^{(k)}$ and $\bm{P}_j^{(k)}$:
\begin{equation}\begin{aligned}
    u(t,\bm{x})\approx\frac{(\Delta q\Delta p)^d}{(2\pi\varepsilon)^{\frac{3d}{2}}}\sum_{(\bm{q},\bm{p})\operatorname{\ on\ grid}}\widetilde{A}(\bm{q},\bm{p})G(\bm{x};\bm{q},\bm{p}),
    \label{eq:wave_reconstruction1}
\end{aligned}\end{equation}
in which
\begin{equation}\begin{aligned}
    \widetilde{A}(\bm{q},\bm{p})=\sum_{\substack{\bm{Q}_j^{(k)}=\bm{q}\\\bm{P}_j^{(k)}=\bm{p}}}A_j e^{\frac{\mathrm{i}}{\varepsilon}S_j}c_j^{(k)}.
    \label{eq:wave_reconstruction2}
\end{aligned}\end{equation}
Now the summation \Cref{eq:wave_reconstruction1} has such a form so that FFT can be utilized directly to enhance efficiency. In summary, we present the pseudo-code for the frozen Gaussian grid-point correction (FGGC) in \Cref{alg:fggc}
\begin{algorithm}\caption{Frozen Gaussian grid-point correction (FGGC) - A basic version.}
\label{alg:fggc}
\begin{algorithmic}
\STATE{\textbf{Step 0: Preparation}}
\STATE{Discrete the meshes, take $\Delta x=O(\varepsilon),\Delta q=O(\sqrt{\varepsilon}),\Delta p=O(\sqrt{\varepsilon})$.}
\STATE{\textbf{Step 1: Initial decomposition}}
\STATE{For each $\bm{q}$, compute $A(0,\bm{q},\bm{p})$ using FFT algorithm by \Cref{eq:initial_decomposition}.}
\STATE{\textbf{Step 2: Time evolution}}
\STATE{For each wave-packet, evolve the ODE system \Cref{eq:2.3}$\sim$\Cref{eq:2.6} up to time $t$, and obtain the FGA variables $A,S,\bm{Q},\bm{P}$.}
\STATE{\textbf{Step 3: Least squares approximation}}
\STATE{Determine the neighbor strategy as in \Cref{eq:neighbor_representation} and \Cref{eq:k_th_neighbor}. For each wave-packet, solve the normal equations \Cref{eq:normal_equation} and obtain the LSA coefficients.}
\STATE{\textbf{Step 4: Wave reconstruction}}
\STATE{Compute the wave function based on \Cref{eq:wave_reconstruction1} and \Cref{eq:wave_reconstruction2}.}
\end{algorithmic}
\end{algorithm}

\newtheorem{remark}{Remark}
\begin{remark} 
    The transform in \Cref{eq:initial_decomposition} consists of a sequence of cut-off FFT. Consequently, one could take $\Delta p=O(\sqrt{\varepsilon})$ as the step size of the phase space mesh. In fact, by the stationary phase approximation, $A(0,\bm{q},\bm{p})$ is localized around the submanifold $\bm{p}=\nabla_{\bm{q}} S_0(\bm{q})$, which inspires one to put the mesh grids of $\bm{p}$ around $\nabla_{\bm{q}} S_0(\bm{q})$ to perform the initial decomposition. Furthermore, taking $\Delta q=O(\sqrt{\varepsilon})$ also ensures the accuracy of the wave function. 
\end{remark}

\begin{remark}
    Normal equations \Cref{eq:normal_equation} can be solved using the pseudo-inverse method. To be specific, consider a rank-deficient LSA problem $ \underset{x}{\text{min}}\|\bm{A}x-\bm{f}\|_2$, then $\bm{A}^\dagger\bm{f}$ is the solution which attains the minimum Euclidean norm among all possible solutions, where $\bm{A}^\dagger$ denotes the pseudo-inverse of $\bm{A}$ and can be obtained by SVD decomposition.
\end{remark}

\begin{remark} The FGGC algorithm is highly parallelizable. For $\textbf{Step 1}$, parallel computation can be performed in the $\bm{q}$ direction. Specifically, to obtain $A(0,\bm{q},\bm{p})$, the $\mathrm{d}\bm{y}$ integral \Cref{eq:initial_decomposition} can be computed separately (using FFT) for each $\bm{q}$. For $\textbf{Step 2}$ and $\textbf{Step 3}$, the evolution and least squares approximation of each wave-packet can also be carried out in parallel. For $\textbf{Step 4}$, the integral $\mathrm{d}\bm{p}$ \Cref{eq:wave_reconstruction1} can be computed separately (using FFT) for each $\bm{q}$, and then sum up to obtain $u(t,\bm{x})$.
\end{remark}

\section{Error Estimate and Efficiency Analysis for FGGC}\label{sec:3}

In this section, we analyze the error and efficiency of the FGGC algorithm. The $L^2$ error of the FGGC wave function $u_\text{fggc}$ with respect to the exact solution $u_\text{exact}$ can be decomposed as follows:
\begin{equation}\begin{aligned}
    \left\|u_\text{exact}-u_\text{fggc}\right\|_{L^2}\leq\left\|u_\text{exact}-u_\text{fga}\right\|_{L^2}+\left\|u_\text{fga}-u_\text{fggc}\right\|_{L^2}.
    \label{eq:3.1}
\end{aligned}\end{equation}
The following theorem shows that the asymptotic error of the FGA ansatz, i.e., the first term on the right-hand side of \Cref{eq:3.1}, is $O(\varepsilon)$:
\newtheorem{theorem}{Theorem}
\begin{theorem} \cite{FGAwave} \label{thm:fgaOeps}
    Assume the potential function $V(\bm{x})\in C^{\infty}(\mathbb{R}^\textup{d})$. Let $u_\textup{exact}$ be the solution of the semi-classical Schr{\"o}dinger equation \Cref{eq:1.1}--\Cref{eq:1.2}, and $u_\textup{FGA}$ be the FGA representation as \Cref{eq:2.1}. For any $t>0$, we have
    \begin{equation}\begin{aligned}
        \left\|u_\textup{exact}(t,\cdot)-u_\textup{fga}(t,\cdot)\right\|_{L^2}\leq C_1(t,d)\varepsilon\left\|u_0\right\|_{L^2},
        \label{eq:3.2}
    \end{aligned}\end{equation}
    where $C_1(t,d)$ is a constant independent of the semi-classical parameter $\varepsilon$.
\end{theorem}

For the second term, we rewrite the difference as
\begin{equation}\begin{aligned}
    u_\text{fga}(t,\cdot)-u_\text{fggc}(t,\cdot)=\frac{(\Delta q\Delta p)^{d}}{(2\pi\varepsilon)^{\frac{3d}{2}}}\sum_{j=1}^{M}A_j\left(G_j(\bm{\cdot})-\sum_{k=1}^n c_j^{(k)}G_j^{(k)}(\bm{\cdot})\right).
    \label{eq:3.3}
\end{aligned}\end{equation}
A direct estimate of this expression is challenging. However, we can estimate the difference between a single wave-packet and its least squares approximation.

\subsection{Error estimate for the least squares approximation and numerical comparison of different strategies}\label{sec:3.1}

We denote the relative $L^2$ error of the least squares approximation for a single Gaussian wave-packet as
\begin{equation}\begin{aligned}
    E_\text{LSA}:=&\frac{\left\|G\left(\bm{x};\bm{Q},\bm{P}\right)-\sum_{k=1}^n c^{(k)}G\left(\bm{x};\bm{\mathcal{N}}(\bm{Q},\bm{P};k)\right)\right\|_{L^2}}{\left\|G(\bm{x};\bm{Q},\bm{P})\right\|_{L^2}}\\
    =&(\pi\varepsilon)^{-\frac{d}{4}}\left\|G(\bm{x};\bm{Q},\bm{P})-\sum_{k=1}^n c^{(k)}G\left(\bm{x};\bm{\mathcal{N}}(\bm{Q},\bm{P};k)\right)\right\|_{L^2}.
    \label{eq:E_LSA}
\end{aligned}\end{equation}
Intuitively, $E_\text{LSA}=E_\text{LSA}(\varepsilon,\bm{Q},\bm{P},\Delta q,\Delta p)$ depends on the following variables:

(1) The semi-classical parameter $\varepsilon$; 

(2) The coordinates of the target $(\bm{Q},\bm{P})$; 

(3) The neighbor strategy, including the number of neighbors $n$ and the shift integers $\delta q_l^{(k)}$ and $\delta p_l^{(k)}$ as in \Cref{eq:neighbor_representation}; 

(4) The size of the mesh steps $\Delta q$ and $\Delta p$. 

We adopt the following conventions for the “mesh strategy” and the “neighbor strategy”: The mesh strategy is considered fixed if 
\begin{equation}\begin{aligned}
    \Delta q = C_q\sqrt{\varepsilon},\ \Delta p = C_p\sqrt{\varepsilon},
    \label{eq:3.5}
\end{aligned}\end{equation}
for some constants $C_q$ and $C_p$. The neighbor strategy is considered fixed if the number of neighbors $n$ and the sequences of shift integers $\bm{\delta q}^{(k)}$ and $\bm{\delta p}^{(k)}$ for $k=1,2,\cdots,n$ are all given. To study the properties of $E_\text{LSA}(\varepsilon,\bm{Q},\bm{P},\Delta q,\Delta p)$, we introduce more notation.

\textbf{Notations.} Denote
\begin{equation}\begin{aligned}
    \operatorname{int}\bm{Q}=\operatorname{floor}\left(\frac{\bm{Q}}{\Delta q}\right)\Delta q,\ \operatorname{int}\bm{P}=\operatorname{floor}\left(\frac{\bm{P}}{\Delta p}\right)\Delta p,
    \label{eq:3.6}
\end{aligned}\end{equation}
and 
\begin{equation}\begin{aligned}
    \operatorname{frac}\bm{Q}=\bm{Q}-\operatorname{int}\bm{Q},\ \operatorname{frac}\bm{P}=\bm{P}-\operatorname{int}\bm{P},
    \label{eq:3.7}
\end{aligned}\end{equation}
where $\operatorname{floor}\left(\cdot\right)$ represents the floor function. Therefore, the phase space coordinate $(\bm{Q},\bm{P})$ is decomposed into an ``integer part'' and a ``fractional part'' where the fractional part always falls within the range $[0,\Delta q)^d\times[0,\Delta p)^d$. We refer to this range as the ``unit hypercube'' in the following text.

Applying transformations to \Cref{eq:E_LSA} leads to the following theorem:

\begin{theorem}\label{thm:LSA_in_one_grid}
    Suppose that the mesh strategy and the neighbor strategy are fixed. Then, for the relative LSA error, we have 
    \begin{equation}\begin{aligned}
        E_\text{LSA}\left(\varepsilon,\bm{Q},\bm{P},\Delta q,\Delta p\right) = E_\text{LSA}\left(\varepsilon,\operatorname{frac}\bm{Q},\operatorname{frac}\bm{P},\Delta q,\Delta p\right).
        \label{eq:3.8}
    \end{aligned}\end{equation}
\end{theorem}

\begin{proof}
    Denote $\left\{c^{(k)}\right\}_{k=1}^n$ as the LSA coefficients with respect to $\left(\bm{Q},\bm{P}\right)$ and $\left\{\widetilde{c}^{(k)}\right\}_{k=1}^n$ as that of $\left(\operatorname{frac}\bm{Q},\operatorname{frac}\bm{P}\right)$. Recall that $\left\{c^{(k)}\right\}_{k=1}^n$ satisfy the normal equations \Cref{eq:normal_equation} and \Cref{eq:normal_equation_Af}, while $\left\{\widetilde{c}^{(k)}\right\}_{k=1}^n$ satisfy
    \begin{equation}\begin{aligned}
        &\sum_{k=1}^n \widetilde{A}_{jk}\widetilde{c}^{(k)}=\widetilde{f}_j\ ,j=1,2,\cdots,n,\\
        \widetilde{A}_{jk}=&\left< G(\bm{x};\bm{\mathcal{N}}(\bm{0},\bm{0};j)),G(\bm{x};\bm{\mathcal{N}}(\bm{0},\bm{0};k))\right> ,\\
        \widetilde{f}_j=&\left<G(\bm{x};\bm{\mathcal{N}}(\bm{0},\bm{0};j)),G(\bm{x};\operatorname{frac}\bm{Q},\operatorname{frac}\bm{P})\right>.
        \label{eq:shift_normal_equation}
    \end{aligned}\end{equation}
    Using \Cref{eq:inner_product}, one gets
    \begin{equation}\begin{aligned}   
        \widetilde{A}_{jk}=A_{jk}\operatorname{exp}\left(-\frac{\mathrm{i}}{\varepsilon}\left(\bm{Q}^{(k)}-\bm{Q}^{(j)}\right)\cdot\operatorname{int}\bm{P}\right),
        \label{eq:A_tilde_and_A}
    \end{aligned}\end{equation}
    and
    \begin{equation}\begin{aligned}   
        \widetilde{f}_{j}=f_{j}\operatorname{exp}\left(-\frac{\mathrm{i}}{\varepsilon}\left(\bm{Q}-\bm{Q}^{(j)}\right)\cdot\operatorname{int}\bm{P}\right).
        \label{eq:f_tilde_and_f}
    \end{aligned}\end{equation}
    So, \Cref{eq:shift_normal_equation} becomes
    \begin{equation}\begin{aligned}   
        \sum_{k=1}^n A_{jk}\operatorname{exp}\left(\frac{\mathrm{i}}{\varepsilon}\left(\bm{Q}-\bm{Q}^{(k)}\right)\cdot\operatorname{int}\bm{P}\right)\widetilde{c}^{(k)}=f_j\ ,j=1,2,\cdots,n.
        \label{eq:3.12}
    \end{aligned}\end{equation}
    Therefore,
    \begin{equation}\begin{aligned}
        c^{(k)}=\widetilde{c}^{(k)}\operatorname{exp}\left(\frac{\mathrm{i}}{\varepsilon}\left(\bm{Q}-\bm{Q}^{(k)}\right)\cdot\operatorname{int}\bm{P}\right).
        \label{eq:shift_coefficient1}
    \end{aligned}\end{equation}    
    Next, perform a change of variable $\bm{x}\to\bm{x}+\operatorname{int}\bm{Q}$, which is equivalent to replacing $\bm{Q}$ with $\operatorname{frac}\bm{Q}$, we have 
    \begin{equation}\begin{aligned}
        &E_\text{LSA}(\varepsilon,\bm{Q},\bm{P},\Delta q,\Delta p) \,(\pi\varepsilon)^{\frac{d}{4}}\\
        =&\left\|G\left(\bm{x};\operatorname{frac}\bm{Q},\bm{P}\right)-\sum_{k=1}^n c^{(k)}G\left(\bm{x};\bm{\mathcal{N}}(\operatorname{frac}\bm{Q},\bm{P};k)\right)\right\|_{L^2}\\
        =&\left\|G(\bm{x};\operatorname{frac}\bm{Q},\operatorname{frac}\bm{P})e^{\frac{\mathrm{i}}{\varepsilon}(\bm{x}-\operatorname{frac}\bm{Q})\cdot\operatorname{int}\bm{P}}-\sum_{k=1}^n c^{(k)}G\left(\bm{x};\bm{\mathcal{N}}(\operatorname{frac}\bm{Q},\bm{P};k)\right)\right\|_{L^2}\\
        =&\left\|G(\bm{x};\operatorname{frac}\bm{Q},\operatorname{frac}\bm{P})-\sum_{k=1}^n c^{(k)}e^{-\frac{\mathrm{i}}{\varepsilon}(\bm{x}-\operatorname{frac}\bm{Q})\cdot\operatorname{int}\bm{P}}G\left(\bm{x};\bm{\mathcal{N}}(\operatorname{frac}\bm{Q},\bm{P};k)\right)\right\|_{L^2}\\
        =&\left\|G(\bm{x};\operatorname{frac}\bm{Q},\operatorname{frac}\bm{P})-\sum_{k=1}^n \widetilde{c}^{(k)}G\left(\bm{x};\bm{\mathcal{N}}(\operatorname{frac}\bm{Q},\operatorname{frac}\bm{P};k)\right)\right\|_{L^2}\\
        =&E_\text{LSA}(\varepsilon,\operatorname{frac}\bm{Q},\operatorname{frac}\bm{P},\Delta q,\Delta p)\,(\pi\varepsilon)^{\frac{d}{4}}.
        \label{eq:3.14}
    \end{aligned}\end{equation}
   Then the statement is proved.
\end{proof}

Using this theorem, the domain of the phase space coordinate can be restricted in the “unit hypercube” when considering the supremum of the relative error, i.e.
\begin{equation}\begin{aligned}
    &\sup_{(\bm{Q},\bm{P})\in\mathbb{R}^{2d}}E_\text{LSA}(\varepsilon,\bm{Q},\bm{P},\Delta q,\Delta p)\\=&\sup_{(\bm{Q},\bm{P})\in\mathbb{R}^{2d}}E_\text{LSA}(\varepsilon,\operatorname{frac}\bm{Q},\operatorname{frac}\bm{P},\Delta q,\Delta p)\\
    =&\sup_{(\bm{Q},\bm{P})\in[0,\Delta q]^d\times[0,\Delta p]^d}E_\text{LSA}(\varepsilon,\bm{Q},\bm{P},\Delta q,\Delta p).
    \label{eq:3.15}
\end{aligned}\end{equation}

Next, we investigate the relationship between $E_\text{LSA}$ and the semi-classical parameter $\varepsilon$.

\begin{theorem}\label{thm:eps_irrelevant}
    Assume that both the mesh strategy and the neighbor strategy are fixed. Then for $\forall \varepsilon_1,\varepsilon_2>0$, we have:
    \begin{equation}\begin{aligned}
        &\sup_{(\bm{Q},\bm{P})\in[0,\Delta q_1]^d\times[0,\Delta p_1]^d}E_\text{LSA}(\varepsilon_1,\bm{Q},\bm{P},\Delta q_1,\Delta p_1)\\
        =&\sup_{(\bm{Q},\bm{P})\in[0,\Delta q_2]^d\times[0,\Delta p_2]^d}E_\text{LSA}(\varepsilon_2,\bm{Q},\bm{P},\Delta q_2,\Delta p_2),
        \label{eq:3.16}
    \end{aligned}\end{equation}
    where $\Delta q_1=C_q\sqrt{\varepsilon_1}$, $\Delta p_1=C_p\sqrt{\varepsilon_1}$ and $\Delta q_2=C_q\sqrt{\varepsilon_2}$, $\Delta p_2=C_p\sqrt{\varepsilon_2}$.
\end{theorem}

\begin{proof}
    It suffices to prove that
    \begin{equation}\begin{aligned}
        E_\text{LSA}(\varepsilon_1,\Delta q_1\bm{s},\Delta p_1\bm{t},\Delta q_1,\Delta p_1)=E_\text{LSA}(\varepsilon_2,\Delta q_2\bm{s},\Delta p_2\bm{t},\Delta q_2,\Delta p_2)
        \label{eq:3.17}
    \end{aligned}\end{equation}
    holds for $\forall (\bm{s},\bm{t})\in[0,1]^d\times[0,1]^d$.
    
    A change of variable $\bm{y}=\frac{\bm{x}}{\sqrt{\varepsilon}}$ in \Cref{eq:E_LSA} yields the LSA error:
    \begin{equation}\begin{aligned}
        &E_\text{LSA}(\varepsilon_1,\Delta q_1\bm{s},\Delta p_1\bm{t},\Delta q_1,\Delta p_1)\\
        &=\pi^{\frac{d}{4}}\left\|G\left(\sqrt{\varepsilon_1}\bm{y};\Delta q_1\bm{s},\Delta p_1\bm{t},\varepsilon_1\right)-\sum_{k=1}^n c_1^{(k)}G\left(\sqrt{\varepsilon_1}\bm{y};\bm{\mathcal{N}}(\Delta q_1\bm{s},\Delta p_1\bm{t};k),\varepsilon_1\right)\right\|_{L^2},
        \label{eq:3.18}
    \end{aligned}\end{equation}
    where $\bm{c}_1$ are the corresponding LSA coefficients .
    
    By \Cref{eq:gaussian_wave_packet}, it holds that 
    \begin{equation}\begin{aligned}
        G\left(\sqrt{\varepsilon}\bm{y};\bm{Q},\bm{P},\varepsilon\right) = G\left(\bm{y};\frac{\bm{Q}}{\sqrt{\varepsilon}},\frac{\bm{P}}{\sqrt{\varepsilon}},1\right).
        \label{eq:3.19}
    \end{aligned}\end{equation}
    Therefore,  
    \begin{equation}\begin{aligned}
        E_\text{LSA}&(\varepsilon_1,\Delta q_1\bm{s},\Delta p_1\bm{t},\Delta q_1,\Delta p_1)\\
        &=\pi^{\frac{d}{4}}\left\|G\left(\bm{y};C_q\bm{s},C_p\bm{t},1\right)-\sum_{k=1}^n c_1^{(k)}G\left(\bm{y};\bm{\mathcal{N}}\left(C_q\bm{s},C_p\bm{t};k\right),1\right)\right\|_{L^2}.
        \label{eq:3.20}
    \end{aligned}\end{equation}
    Similarly, 
    \begin{equation}\begin{aligned}
        E_\text{LSA}&(\varepsilon_2,\Delta q_2\bm{s},\Delta p_2\bm{t},\Delta q_2,\Delta p_2)\\
        &=\pi^{\frac{d}{4}}\left\|G\left(\bm{y};C_q\bm{s},C_p\bm{t},1\right)-\sum_{k=1}^n c_2^{(k)}G\left(\bm{y};\bm{\mathcal{N}}\left(C_q\bm{s},C_p\bm{t};k\right),1\right)\right\|_{L^2},
        \label{eq:3.21}
    \end{aligned}\end{equation}
    where $\bm{c}_2$ are the corresponding LSA coefficients.
    
    By comparing \Cref{eq:3.20} and \Cref{eq:3.21}, it suffices to show that $\bm{c}_1=\bm{c}_2$. Substituting the neighbor representation \Cref{eq:neighbor_representation} into the normal equations \Cref{eq:normal_equation} and \Cref{eq:normal_equation_Af}, one obtains the coefficient matrix of the normal equations for $\bm{c}_1$:
    \begin{equation}\begin{aligned}
        A_{jk}=\operatorname{exp}\bigg(&-\frac{C_q^2}{4}\left|\bm{\delta q}^{(j)}-\bm{\delta q}^{(k)}\right|^2-\frac{C_p^2}{4}\left|\bm{\delta p}^{(j)}-\bm{\delta p}^{(k)}\right|^2 \\
        &+\frac{\mathrm{i}C_qC_p}{2}\left(\bm{\delta q}^{(k)}-\bm{\delta q}^{(j)}\right)\cdot\left(\bm{\delta p}^{(j)}+\bm{\delta p}^{(k)}\right)\bigg),
        \label{eq:Ajk_no_epsilon}
    \end{aligned}\end{equation}
    and the right-hand side vector:
    \begin{equation}\begin{aligned}
        f_j=\operatorname{exp}\bigg(&-\frac{C_q^2}{4}\left|\bm{\delta q}^{(j)}-\bm{s}\right|^2-\frac{C_p^2}{4}\left|\bm{\delta p}^{(j)}-\bm{t}\right|^2 \\
        &+\frac{\mathrm{i}C_q C_p}{2}\left(\bm{s}-\bm{\delta q}^{(j)}\right)\cdot\left(\bm{\delta p}^{(j)}+\bm{t}\right)\bigg).
        \label{eq:fj_no_epsilon}
    \end{aligned}\end{equation}
    By performing similar derivations, it is obvious that the normal equations corresponding to $\bm{c}_2$ have exactly the same expressions as \Cref{eq:Ajk_no_epsilon} and \Cref{eq:fj_no_epsilon}, since the above two equations do not contain explicit dependence on $\varepsilon_1$. Consequently, it follows that $\bm{c}_1=\bm{c}_2$, which validates \Cref{eq:3.17}, thus proving the statement.
\end{proof}

Once the discrete mesh strategy and the neighbor strategy are determined, denote the supremum of the relative LSA error $E_\text{LSA}$ over all values of $\varepsilon$ and all wave-packets as
\begin{equation}\begin{aligned}
    M_\text{LSA}:=&\sup_{\substack{\varepsilon>0\\(\bm{Q},\bm{P})\in\mathbb{R}^{2d}}}E_\text{LSA}\left(\varepsilon,\bm{Q},\bm{P},\Delta q,\Delta p\right),
    \label{eq:M_LSA_def}
\end{aligned}\end{equation}
where $\Delta q=C_q\sqrt{\varepsilon}$ and $\Delta p=C_p\sqrt{\varepsilon}$. Then \Cref{thm:LSA_in_one_grid} and \Cref{thm:eps_irrelevant} bring us closer to achieving our goal of determining $M_\text{LSA}$:
\begin{equation}\begin{aligned}
    M_\text{LSA}=&\sup_{\substack{\varepsilon>0 \\(\bm{Q},\bm{P})\in[0,\Delta q]^d\times[0,\Delta p]^d}} E_\text{LSA}\left(\varepsilon,\bm{Q},\bm{P},\Delta q,\Delta p\right)\\
    =&\sup_{(\bm{Q},\bm{P})\in[0,\Delta q]^d\times[0,\Delta p]^d}E_\text{LSA}\left(\varepsilon_0,\bm{Q},\bm{P},\Delta q,\Delta p\right)
    \label{eq:M_LSA}
\end{aligned}\end{equation}
where $\varepsilon_0>0$ can be chosen arbitrarily. It is clear that $M_\text{LSA}$ is bounded since the supremum is taken over a compact set with respect to $(\bm{Q},\bm{P})$, and depends only on the mesh strategy and the neighbor strategy. At this point, we can assert that the least squares approximation error for a single wave-packet is independent of $\varepsilon$ under the determined mesh strategy and the neighbor strategy.

We then proceed to investigate the numerical behavior of $M_\text{LSA}$ in various mesh strategies and neighbor strategies. Take $d=1$, $\varepsilon=1/64$, $\Delta q=2\sqrt{\varepsilon}$, $\sqrt{\varepsilon}$, $\sqrt{\varepsilon}/2$, $\sqrt{\varepsilon}/4$, $\Delta p=\pi\sqrt{\varepsilon}/2$, $\pi\sqrt{\varepsilon}/4$, $\pi\sqrt{\varepsilon}/8$, $\pi\sqrt{\varepsilon}/16$.  Consider several different neighbor strategies as illustrated in \Cref{fig:3.2neighbors}.  We refer to these strategies as \textbf{Q2P2}, \textbf{Q2P4}, \textbf{Q4P2}, and \textbf{Q4P4}, where \textbf{Q$m$P$n$} denotes a tensor grid set with $m$ nearest $\bm{Q}$ neighbors and $n$ nearest $\bm{P}$ neighbors. 

We compute \Cref{eq:M_LSA} numerically for different strategies, and the results are shown in Tables~\ref{tab:Q2P2}--\ref{tab:Q4P4}. It can be observed that both decreasing the step size and increasing the number of neighbors contribute to a reduction in the LSA error, which is consistent with intuitive expectations. 

\begin{figure}[htbp]
    \centering
    \includegraphics[width=1\linewidth]{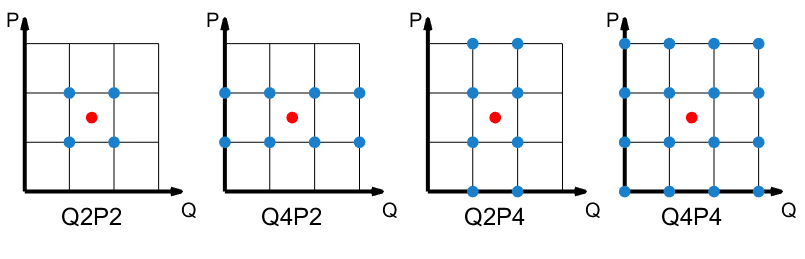}
    \caption{Four commonly used neighbor strategies in one dimension. Red: the target wave-packet. Blue: the neighboring wave-packets to approximate the target wave-packet.}
    \label{fig:3.2neighbors}
\end{figure}

\begin{table}[htbp]
    \centering
    \begin{minipage}{0.95\textwidth}
        \centering
        \begin{tabular}{|c|c|c|c|c|}
            \hline
            \diagbox{$dp$}{$dq$} & $2\sqrt{\varepsilon}$ & $\sqrt{\varepsilon}$ & $\sqrt{\varepsilon}/2$ & $\sqrt{\varepsilon}/4$ \\ \hline
            $\pi\sqrt{\varepsilon}/2$  & 1.540E-01 & 4.815E-02 & 2.672E-02 & 2.112E-02 \\ \hline
            $\pi\sqrt{\varepsilon}/4$  & 7.831E-02 & 1.072E-02 & 3.141E-03 & 1.699E-03 \\ \hline
            $\pi\sqrt{\varepsilon}/8$  & 5.705E-02 & 5.111E-03 & 6.859E-04 & 1.982E-04 \\ \hline
            $\pi\sqrt{\varepsilon}/16$ & 5.159E-02 & 3.669E-03 & 3.217E-04 & 4.312E-05 \\ \hline
        \end{tabular}
        \caption{Values of $M_\text{LSA}$ under the neighbor strategy \textbf{Q2P2}.}
        \label{tab:Q2P2}
    \end{minipage}
    
    \vspace{1em} 
    \begin{minipage}{0.95\textwidth}
        \centering
        \begin{tabular}{|c|c|c|c|c|}
            \hline
            \diagbox{$dp$}{$dq$} & $2\sqrt{\varepsilon}$ & $\sqrt{\varepsilon}$ & $\sqrt{\varepsilon}/2$ & $\sqrt{\varepsilon}/4$ \\ \hline
            $\pi\sqrt{\varepsilon}/2$  & 3.447E-02 & 9.531E-03 & 4.683E-03 & 3.518E-03 \\ \hline
            $\pi\sqrt{\varepsilon}/4$  & 2.722E-03 & 1.728E-04 & 4.318E-05 & 2.071E-05 \\ \hline
            $\pi\sqrt{\varepsilon}/8$  & 6.586E-04 & 1.093E-05 & 6.700E-06 & 1.629E-06 \\ \hline
            $\pi\sqrt{\varepsilon}/16$ & 3.819E-04 & 6.017E-06 & 5.660E-06 & 1.165E-06 \\ \hline
        \end{tabular}
        \caption{Values of $M_\text{LSA}$ under the neighbor strategy \textbf{Q4P2}.}
        \label{tab:Q4P2}
    \end{minipage}
    
    \vspace{1em} 
    \begin{minipage}{0.95\textwidth}
        \centering
        \begin{tabular}{|c|c|c|c|c|}
            \hline
            \diagbox{$dp$}{$dq$} & $2\sqrt{\varepsilon}$ & $\sqrt{\varepsilon}$ & $\sqrt{\varepsilon}/2$ & $\sqrt{\varepsilon}/4$ \\ \hline
            $\pi\sqrt{\varepsilon}/2$  & 7.320E-02 & 2.183E-03 & 2.297E-04 & 7.624E-05 \\ \hline
            $\pi\sqrt{\varepsilon}/4$  & 3.114E-02 & 4.129E-04 & 8.925E-06 & 4.336E-06 \\ \hline
            $\pi\sqrt{\varepsilon}/8$  & 2.159E-02 & 1.653E-04 & 1.262E-05 & 7.480E-06 \\ \hline
            $\pi\sqrt{\varepsilon}/16$ & 1.928E-02 & 1.130E-04 & 4.329E-06 & 2.171E-06 \\ \hline
        \end{tabular}
        \caption{Values of $M_\text{LSA}$ under the neighbor strategy \textbf{Q2P4}.}
        \label{tab:Q2P4}
    \end{minipage}
    
    \vspace{1em} 
    \begin{minipage}{0.95\textwidth}
        \centering
        \begin{tabular}{|c|c|c|c|c|}
            \hline
            \diagbox{$dp$}{$dq$} & $2\sqrt{\varepsilon}$ & $\sqrt{\varepsilon}$ & $\sqrt{\varepsilon}/2$ & $\sqrt{\varepsilon}/4$ \\ \hline
            $\pi\sqrt{\varepsilon}/2$  & 2.054E-03 & 1.326E-05 & 2.365E-05 & 1.178E-05 \\ \hline
            $\pi\sqrt{\varepsilon}/4$  & 6.237E-05 & 1.385E-05 & 1.663E-06 & 3.004E-06 \\ \hline
            $\pi\sqrt{\varepsilon}/8$  & 1.719E-05 & 4.352E-06 & 1.125E-06 & 7.155E-07 \\ \hline
            $\pi\sqrt{\varepsilon}/16$ & 1.596E-04 & 1.734E-05 & 1.428E-06 & 2.181E-07 \\ \hline
        \end{tabular}
        \caption{Values of $M_\text{LSA}$ under the neighbor strategy \textbf{Q4P4}.}
        \label{tab:Q4P4}
    \end{minipage}
\end{table}

\begin{remark}
    The mesh strategy and the neighbor strategy are critical to the accuracy-efficiency trade-off. Taking more neighbor points or finer meshes may lead to higher accuracy, but also cost relatively more computational resources. For the mesh strategy, we recommend adopting $\Delta q=\sqrt{\varepsilon}/2,\Delta p=\pi\sqrt{\varepsilon}/8$, which is empirically validated, feasible, and achieves sufficient accuracy. For the neighbor strategy, the number of neighbors in the above examples varies from $2^{2d}$ to $2^{4d}$, which is a sufficiently large number when $d$ grows. We suggest that one could take the mesh strategy \textbf{Q2P2} when $d=3$, at which point the computational complexity will be high and the accuracy requirement can be slightly lowered, and take \textbf{Q4P2} when $d=1,2$, since the LSA error converges rapidly and is negligible under such a strategy. As for higher-dimensional cases, one should elaborately design the neighbor strategy to balance the efficiency and accuracy, considering that the number of neighbors introduced here grows exponentially, and many of the neighbors may not contribute to improving the error of the least square approximation. For example, one may select only the nearest several neighbor grid points to perform the least square approximation.
\end{remark}

\begin{remark}
    The $L^2$ error estimate for the least squares approximation can be extended to the $H^1$ norm, and the resulting error bound remains independent of $\varepsilon$. Specifically, let $G(\bm{x};\bm{Q},\bm{P})$ and its least squares approximation $\sum_{k=1}^n c^{(k)} G(\bm{x};\bm{Q}^{(k)},\bm{P}^{(k)})$ be defined as above, with the mesh and neighbor strategies fixed. For simplicity, we denote $G^{(k)}(\bm{x}):= G(\bm{x};\bm{Q}^{(k)},\bm{P}^{(k)})$. Then there exists a constant $M_\text{LSA}^{(1)}$, independent of $\varepsilon$ and $(\bm{Q},\bm{P})$, such that
        \begin{equation}\begin{aligned}
            \left\|G\left(\cdot\right)-\sum_{k=1}^n c^{(k)}G^{(k)}\left(\cdot\right)\right\|_{H^1}\leq M_\text{LSA}^{(1)}\left\|G(\cdot)\right\|_{H^1}.
            \label{eq:3.add1}
        \end{aligned}\end{equation}
    This follows by gradient estimates obtained via the triangle inequality
    \begin{equation}\begin{aligned}
        &\left\|\nabla G-\sum_{k=1}^n c^{(k)}\nabla G^{(k)}\right\|_{L^2}
        \leq I_1+I_2+I_3,
        \label{eq:3.add5}
    \end{aligned}\end{equation}    
    where the three terms on the right side of the above inequality are defined and estimated as follows:
    \begin{equation}\begin{aligned}
        I_1
        :=\left\|\frac{\mathrm{i}}{\varepsilon}\bm{P}\left(G-\sum_{k=1}^n c^{(k)}G^{(k)}\right)\right\|_{L^2}
        \leq\frac{|\bm{P}|}{\varepsilon}M_\text{LSA}(\pi\varepsilon)^{\frac{d}{4}}\leq M_\text{LSA}\left\|G\right\|_{H^1},
        \label{eq:3.add6}
    \end{aligned}\end{equation}
    \begin{equation}\begin{aligned}
        I_2
        :=\left\|\frac{(\bm{x}-\bm{Q})}{\varepsilon}\left(G-\sum_{k=1}^n c^{(k)}G^{(k)}\right)\right\|_{L^2}
        \leq\frac{C}{\sqrt{\varepsilon}}M_\text{LSA}(\pi\varepsilon)^{\frac{d}{4}}\leq CM_\text{LSA}\left\|G\right\|_{H^1},
        \label{eq:3.add9}
    \end{aligned}\end{equation}
    and
    \begin{equation}\begin{aligned}
        I_3
        &:=\left\|\sum_{k=1}^n c^{(k)}\left(\frac{\mathrm{i}}{\varepsilon}\left(\bm{P}-\bm{P}^{(k)}\right)+\frac{\mathrm{1}}{\varepsilon}\left(\bm{Q}-\bm{Q}^{(k)}\right)\right)G^{(k)}\right\|_{L^2}\\
        &\leq\sum_{k=1}^n \left|c^{(k)}\right| \left|{\frac{\mathrm{i}}{\varepsilon}\left(\bm{P}-\bm{P}^{(k)}\right)+\frac{\mathrm{1}}{\varepsilon}\left(\bm{Q}-\bm{Q}^{(k)}\right)}\right| \left\|G^{(k)}\right\|_{L^2}\\
        &\leq\sum_{k=1}^n \left|c^{(k)}\right|\frac{C(C_q+C_p)}{\sqrt{\varepsilon}}(\pi\varepsilon)^{\frac{d}{4}}\leq\sum_{k=1}^n \left|c^{(k)}\right|{C(C_q+C_p)}\left\|G\right\|_{H^1}.
        \label{eq:3.add7}
    \end{aligned}\end{equation}
    Since all these bounding constants are independent of $\varepsilon$ and $(\bm{Q},\bm{P})$, the statement holds. 
    
    To investigate the numerical behavior of $M^{(1)}_\text{LSA}$, we compute 
    \begin{equation}\begin{aligned}
        \sup_{(\bm{Q},\bm{P})\in[0,\Delta q]^d\times[0,\Delta p]^d}\frac{\left\|\nabla G-\sum_{k=1}^n c^{(k)}\nabla G^{(k)}\right\|_{L^2}}{\left\|G\right\|_{H^1}}
        \label{eq:3.add8}
    \end{aligned}\end{equation}
    for different values of $\varepsilon$ and neighbor strategies, with the mesh strategy fixed as $\Delta q=\sqrt{\varepsilon}/2$ and $\Delta p=\pi\sqrt{\varepsilon}/8$. Take $d=1$, $\varepsilon=1/2^6,1/2^8,\cdots,1/2^{16}$ and neighbor strategies \textbf{Q2P2}, \textbf{Q2P4}, \textbf{Q4P2} and \textbf{Q4P4}. The results are shown in \Cref{tab:MLSA1}. It can be observed that the value of \Cref{eq:3.add8} converges when $\varepsilon\to 0^+$, which confirms the statement above. 
    From the error data in \Cref{tab:MLSA1}, we observe that the $H^1$ error decreases as the number of neighbors increases, consistent with the behavior in the $L^2$ case. While the numerical results suggest that the $H^1$ error is independent of $\varepsilon$ and can be made arbitrarily small by selecting a suitable neighbor strategy, it is not clear whether a sharp bound in terms of $M_\text{LSA}$ holds, as the constants in the estimate \Cref{eq:3.add7} may depend on the mesh and neighbor strategies. Therefore, we conservatively conclude that the $H^1$ error of the least squares approximation is independent of $\varepsilon$ and can be effectively controlled by an appropriate choice of neighbor strategy, although a precise bound in terms of $M_\text{LSA}$ is not established.

    \begin{table}[htbp]
    \centering
    \vspace{1em} 
        \centering
        \begin{tabular}{|c|c|c|c|c|}
            \hline
            \diagbox{$\varepsilon$}{$\text{NS}$} & $\textbf{Q2P2}$ & $\textbf{Q4P2}$ & $\textbf{Q2P4}$ & $\textbf{Q4P4}$ \\ \hline
            $1/2^{6}$   & 2.023E-03 & 4.818E-05 & 2.560E-05 & 4.856E-06 \\ \hline
            $1/2^{8}$   & 2.046E-03 & 4.874E-05 & 2.590E-05 & 4.912E-06 \\ \hline
            $1/2^{10}$  & 2.052E-03 & 4.888E-05 & 2.598E-05 & 4.926E-06 \\ \hline
            $1/2^{12}$  & 2.054E-03 & 4.892E-05 & 2.599E-05 & 4.930E-06 \\ \hline
            $1/2^{14}$  & 2.054E-03 & 4.893E-05 & 2.600E-05 & 4.931E-06 \\ \hline
            $1/2^{16}$  & 2.054E-03 & 4.893E-05 & 2.600E-05 & 4.931E-06 \\ \hline
        \end{tabular}
        \caption{Values of \Cref{eq:3.add8} for different $\varepsilon$ values and neighbor strategies.}
        \label{tab:MLSA1}
\end{table}

\end{remark}

\subsection{Improved versions of FGGC}\label{sec:3.2}

In $\textbf{Step 3}$ of \Cref{alg:fggc}, performing the least squares approximation for each wave-packet involves computing the pseudo-inverse $\bm{A}^\dagger$ and the matrix-vector multiplication $\bm{A}^\dagger\bm{f}$. This requires $O(n^3+n^2)$ operations for each wave-packet. However, recalling \Cref{eq:shift_normal_equation} and \Cref{eq:shift_coefficient1} from the proof of \Cref{thm:LSA_in_one_grid}, we see that approximating $G(\bm{x};\bm{Q},\bm{P})$ using $G\left(\bm{x};\bm{\mathcal{N}}(\bm{Q},\bm{P};k)\right)$ is almost equivalent to approximating $G\left(\bm{x};\operatorname{frac}\bm{Q},\operatorname{frac}\bm{P}\right)$ using $G\left(\bm{x};\bm{\mathcal{N}}(\bm{0},\bm{0};k)\right)$, differing only by a constant factor in the approximation coefficients. It should be noted that, as the second equation in \Cref{eq:shift_normal_equation} shows, $\widetilde{A}_{jk}$ is fixed with given mesh strategy and neighbor strategy, so $\widetilde{\bm{A}}^\dagger$ can be pre-computed independently. In doing so, to obtain the LSA coefficients for each off-grid wave-packet, it suffices to calculate $\widetilde{\bm{f}}$ and $\widetilde{\bm{c}}=\widetilde{\bm{A}}^\dagger\widetilde{\bm{f}}$, and then \Cref{eq:shift_coefficient1} yields that the LSA coefficients $\left\{c^{(k)}\right\}_{k=1}^n$ can be obtained from $\left\{\widetilde{c}^{(k)}\right\}_{k=1}^n$. 
 
By pre-computing $\widetilde{\bm{A}}^\dagger$, the LSA process requires only $O\left(n^3\right)+O\left(Mn^2\right)$ operations instead of $O\left(Mn^3\right)+O\left(Mn^2\right)$ and reduces the computational cost. Consequently, we propose an improved version of the FGGC algorithm, presented in \Cref{alg:fggcgood}.

\begin{algorithm}\caption{Frozen Gaussian grid-point correction (FGGC) - An improved version.}
\label{alg:fggcgood}
\begin{algorithmic}
\STATE{\textbf{Step 0: Preparation}}
\STATE{Determine the mesh strategy $\Delta x=O(\varepsilon),\Delta q=O\left(\sqrt{\varepsilon}\right),\Delta p=O\left(\sqrt{\varepsilon}\right)$.}
\STATE{Determine the neighbor strategy, including $n$ and $\bm{\delta q}^{(k)}$, $\bm{\delta p}^{(k)}$.}
\STATE{Pre-compute $\widetilde{\bm{A}}^\dagger$.}
\STATE{\textbf{Step 1: Initial decomposition}}
\STATE{For each $\bm{q}$, compute $A(0,\bm{q},\bm{p})$ by \Cref{eq:initial_decomposition}.}
\STATE{\textbf{Step 2: Time evolution}}
\STATE{For each wave-packet, evolve the ODE system \Cref{eq:2.3}$\sim$\Cref{eq:2.6} up to time $t$, and obtain the FGA variables $A,S,\bm{Q},\bm{P}$.}
\STATE{\textbf{Step 3: Least squares approximation}}
\STATE{For each wave-packet, calculate $\widetilde{\bm{f}}$ as \Cref{eq:shift_normal_equation}, let $\widetilde{\bm{c}}=\widetilde{\bm{A}}^\dagger\widetilde{\bm{f}}$, and then calculate $\bm{c}$ as \Cref{eq:shift_coefficient1}.}
\STATE{\textbf{Step 4: Wave reconstruction}}
\STATE{Compute the wave function based on \Cref{eq:wave_reconstruction1} and \Cref{eq:wave_reconstruction2}.}
\end{algorithmic}
\end{algorithm}

Furthermore, motivated by the following two observations, we propose a multi-step version of the FGGC algorithm: 

(1) Most of the computation time is concentrated in the ODE evolution step when the final time is large enough or $\varepsilon$ is not small enough; 

(2) The FGGC evolves a set of on-grid wave-packets to another set of on-grid wave-packets. This process can be viewed as a linear map and can be learned by the solver. 

Specifically, suppose $T_\text{final}=T_\text{multi}T_\text{evo}$, which means running the algorithm $T_\text{multi}$ times and each time evolving the wave function up to time $T_\text{evo}$. For a fixed time $T_\text{evo}$, a single on-grid wave-packet before evolution corresponds to $n$ on-grid wave-packets after evolution, each carrying a weight. This process can be viewed as a linear map, which takes an initial position as input and produces $n$ final positions and $n$ coefficients as output. For simplicity, the $n$ final positions can also be represented by a single final position once the neighbor strategy is determined. At each iteration, the algorithm begins with the initial decomposition. Then, the solver learns the evolution and splitting results consisting of $n+2$ triples (initial position, final position, and $n$ splitting coefficients) and improves efficiency in the subsequent runs. Once all the wave-packets have been computed, the wave function can be reconstructed, completing one iteration of the algorithm. Repetition of the above procedure $T_\text{multi}$ times leads to the final state of the wave function. The multi-step version of FGGC is detailed in \Cref{alg:fggc_multistep}.

\begin{algorithm}\caption{Frozen Gaussian grid-point correction (FGGC) - A multi-step version.}
\label{alg:fggc_multistep}
\begin{algorithmic}
\STATE{\textbf{Step 0: Preparation}}
\STATE{Determine the mesh strategy $\Delta x=O(\varepsilon),\Delta q=O\left(\sqrt{\varepsilon}\right),\Delta p=O\left(\sqrt{\varepsilon}\right)$.}
\STATE{Determine the neighbor strategy, including $n$ and $\bm{\delta q}^{(k)}$, $\bm{\delta p}^{(k)}$.}
\STATE{Pre-compute $\widetilde{\bm{A}}^\dagger$.}
\FOR{$N_\text{T} = 1,\dots,T_\text{multi}$}
\STATE{\textbf{Step 1: Initial decomposition}}
\STATE{For each $\bm{q}$, compute $A(0,\bm{q},\bm{p})$ by \Cref{eq:initial_decomposition}, in which $u(0,\bm{x})$ should be substituted by $u((N_\text{T}-1)T_\text{evo},\bm{x})$.}
\STATE{\textbf{Step 2: Time evolution and least squares approximation}}
\STATE{For each wave-packet not learned yet: evolve the ODE system \Cref{eq:2.3}$\sim$\Cref{eq:2.6} up to time $T_\text{evo}$, calculate the approximation coefficients $\bm{c}$ using the pre-computed $\widetilde{\bm{A}}^\dagger$. The relationship between the initial position, final position and splitting coefficients is learned during the process.}
\STATE{\textbf{Step 3: Wave reconstruction}}
\STATE{Compute the wave function $u(N_\text{T}T_\text{evo},\bm{x})$ based on \Cref{eq:wave_reconstruction1} and \Cref{eq:wave_reconstruction2}.}
\ENDFOR
\end{algorithmic}
\end{algorithm}

The advantages of this multi-step solver are: 

(1) One may only need to evolve a larger number of points during the first run of the algorithm, and subsequently evolve much fewer points in later runs, which may improve efficiency compared to evolving all points up to the final time. 

(2) The solver may eventually learn the entire map of the grid points as the algorithm runs sufficiently many times, and by then, there would be no need to calculate the ODE evolution and LSA anymore. 

(3) More information about the wave function at intermediate stages can be obtained. 

However, this design also has its disadvantages: 

(1) The error accumulates with each run. 

(2) The algorithm is more efficient only if the ODE evolution step dominates the entire computation time of the algorithm.

\section{Numerical Examples}\label{sec:4}
In this section, we present several numerical examples to verify the accuracy and efficiency of the algorithm. Our objectives are as follows: 

(1) To demonstrate numerically that the error introduced by the least squares approximation is independent of $\varepsilon$, by calculating $\left\|u_\textup{FGA}-u_\textup{FGGC}\right\|_{L^2}$ and observing its trend as $\varepsilon$ varies.

(2) To compare the efficiency improvements of FGGC with FGA and TSSP.

(3) To investigate how different neighbor strategies affect the performance of the FGGC algorithm. 

All numerical experiments are carried out on a system equipped with a 96-core AMD Ryzen Threadripper 7995WX processor (384MB of L3 cache), 256GB of RAM, running Ubuntu 24.04. Parallel implementations are compiled using the Intel IFX Fortran compiler 2025.0.4, and parallelization is achieved via OpenMP. All CPU runtimes reported in the following subsections are in seconds.

\subsection{One-dimensional case}\label{section:4.1}
\newtheorem{example}{Example}
\begin{example}
    Take
    \begin{equation}\begin{aligned}
        u_0(x)=\left(\frac{64}{\pi}\right)^{\frac{1}{4}}\operatorname{exp}\left(-32x^2+\frac{\mathrm{i}}{\varepsilon}x\right)\text{\ \ and\ \ } V(x)=1-\cos(\pi x).
        \label{eq:ex11}
    \end{aligned}\end{equation}
    We solve in the spatial interval $[-2,2]$ with the final time $T=0.8s$.
    \label{ex:ex1}
\end{example}

We compute the FGGC, FGA, and TSSP solutions, using the latter two as reference benchmarks. The mesh parameters are set as $\Delta x=\varepsilon$, $\Delta q=\sqrt{\varepsilon}/2$, $\Delta p=\pi\sqrt{\varepsilon}/8$ for both FGGC and FGA, and the ODE system is solved using the 4th-order Runge-Kutta method with time step $\Delta t=10^{-4}$. The TSSP solution uses $\Delta x=\varepsilon$ and $\Delta t=10^{-4}$.

The $L^2$ errors for FGA and FGGC with different neighbor strategies are shown in \Cref{tab:ex1_totalerror_V1}, using TSSP as the reference. It is observed that FGGC achieves nearly the same overall error as FGA. The error introduced by the LSA is reported in \Cref{tab:ex1_lsaerror_V1}, with FGA being the reference. The LSA error is sufficiently small and decreases rapidly as the neighbor strategy is refined. Furthermore, we examine the dependence of the error on $\varepsilon$. As shown in \Cref{tab:ex1_lsaerror_V1}, the LSA error of FGGC remains almost unchanged as $\varepsilon$ varies, consistent with the theoretical analysis in \Cref{sec:3.1}.

Both FGA and FGGC are highly parallelizable. Although TSSP can also be parallelized, its efficiency may be limited by the ``butterfly operations'' in the FFT algorithm. The runtimes of the three algorithms in sequential and parallel scenarios are presented in \Cref{tab:ex1_sequtime} and \Cref{tab:ex1_paratime}, respectively. For large $\varepsilon$, the runtime improvement of FGGC over FGA is not significant, as most of the computational cost is spent on solving the ODE system, making the reconstruction step negligible. However, as $\varepsilon$ decreases, the efficiency advantage of FGGC becomes increasingly evident. Compared to TSSP, both FGA and FGGC are more efficient, especially in parallel computing scenarios.

\subsection{Two-dimensional case}\label{section:4.2}
\begin{example}
    Take
    \begin{equation}\begin{aligned}
        u_0(x,y)=\left(\frac{64}{\pi}\right)^{\frac{1}{2}}\operatorname{exp}\left(-32(x^2+y^2)+\frac{\mathrm{i}}{\varepsilon}(x+y)\right),
        \label{eq:ex21}
    \end{aligned}\end{equation}
    and
    \begin{equation}\begin{aligned}
        V_1(x,y)=1-e^{x^2+y^2}\text{\ \ or\ \ }V_2(x,y)=\frac{x^2+y^2}{2}.
        \label{eq:ex22}
    \end{aligned}\end{equation}
    We solve in the spatial interval $[-2,2]^2$ with the final time $T=0.8s$.
    \label{ex:ex2}
\end{example}

Similarly to \Cref{ex:ex1}, we set $\Delta x=\varepsilon$, $\Delta q=\sqrt{\varepsilon}/2$, $\Delta p=\pi\sqrt{\varepsilon}/8$, and $\Delta t=10^{-4}$ for FGGC and FGA, and $\Delta x=\varepsilon$ and $\Delta t=10^{-4}$ for TSSP. The $L^2$ errors for FGA and FGGC solutions are reported in \Cref{tab:ex2_totalerror_V1} and \Cref{tab:ex2_totalerror_V2}. The LSA errors are reported in \Cref{tab:ex2_LSAerror_V1} and \Cref{tab:ex2_LSAerror_V2}. These results indicate that the LSA error is also sufficiently small, even compared to the discretization or truncation errors. Thus, the error introduced by the LSA is negligible, independent of $\varepsilon$, and decreases rapidly as the number of neighbor points increases.

For the two-dimensional case, all algorithms are run in parallel. In the case of potential $V_1$, the runtimes of the three algorithms are reported in \Cref{tab:ex2_totaltime}. The most notable finding is that FGGC dramatically reduces the computational cost of the reconstruction step, making it several times faster than FGA. Although the least squares approximation introduces some additional cost, it is negligible compared to the significant time savings in reconstruction. When $\varepsilon$ is sufficiently small, TSSP becomes extremely time-consuming (e.g., about 3 days for $\varepsilon=2^{-14}$), while the parallelized FGGC can produce reliable results in less than 8 minutes.

\subsection{Three-dimensional case}\label{section:4.3}
\begin{example}
    Take 
    \begin{equation}\begin{aligned}
        u_0(x,y,z)=\left(\frac{64}{\pi}\right)^\frac{3}{4}\operatorname{exp}\left(-32(x^2+y^2+z^2)+\frac{\mathrm{i}}{\varepsilon}(x+y+z)\right),
        \label{eq:ex31}
    \end{aligned}\end{equation}
    and
    \begin{equation}\begin{aligned}
        V(x,y,z)=\frac{x^2+y^2+z^2}{2}.
        \label{eq:ex32}
    \end{aligned}\end{equation}
    We solve in the spatial interval $[-2,2]^3$ with the final time $T=0.8s$.
    \label{ex:ex3}
\end{example}

For the three-dimensional cases, the mesh steps for FGGC and FGA are set as follows: Take $\Delta x=\varepsilon$, $\Delta t=10^{-4}$. For $\varepsilon=1/2^5$ or $1/2^7$, set $\Delta q=\sqrt{\varepsilon/2}/2$ and $\Delta p=\pi\sqrt{\varepsilon/2}/8$; For $\varepsilon=1/2^6$ or $1/2^8$, set $\Delta q=\sqrt{\varepsilon}/2$ and $\Delta p=\pi\sqrt{\varepsilon}/8$. TSSP uses $\Delta x=\varepsilon$ and $\Delta t=10^{-4}$. The $L^2$ errors for FGA and FGGC are shown in \Cref{tab:ex3_totalerror}, and the LSA errors for FGGC are in \Cref{tab:ex3_LSAerror}. 
It can be seen that the LSA error remains small, independent of $\varepsilon$, and decreases rapidly with more neighbor points.

We further analyze the efficiency of the three algorithms in parallel scenario. The runtimes are reported in \Cref{tab:ex3_totaltime}. Again, FGGC significantly reduces the computational cost of the reconstruction step. In three dimensions, solving the Schr{\"o}dinger equation is extremely challenging: both TSSP and FGA become prohibitively expensive as $\varepsilon$ decreases, while FGGC remains practical due to its efficient reconstruction.

\subsection{Multi-step FGGC Algorithm}\label{section:4.4}

In this subsection, we investigate the error and efficiency of the multi-step FGGC algorithm (see \Cref{alg:fggc_multistep}).  The first example is given in one dimension:

\begin{example}
    Take 
    \begin{equation}\begin{aligned}
        u_0(x)=\sqrt[4]{\frac{64}{\pi}}\operatorname{exp}\left(-32x^2+\frac{\mathrm{i}}{\varepsilon}x\right),
        \label{eq:ex41}
    \end{aligned}\end{equation}
    and 
    \begin{equation}\begin{aligned}
        V_1(x)=1-\cos(\pi x)\text{\ \ or\ \ }V_2(x)=\frac{x^2}{2}
        \label{eq:ex42}
    \end{aligned}\end{equation}
    We solve in the spatial interval $[-2,2]$ with $\varepsilon=1/2^{12}$ and the final time $T=0.8s$, using the multi-step version FGGC.
    \label{ex:ex4}
\end{example}

We use the same mesh steps as in \Cref{section:4.1} and the neighbor strategy \textbf{Q2P2}. Both potential functions in this example both form a potential well, so the wave-packets remain confined during evolution. With $T=0.8s$ and $k=10^{-4}$, the ODE system requires $8000$ steps. We decompose these $8000$ steps into several scenarios: $4000\times2$, $2000\times4$, $1000\times8$, $500\times16$, and $250\times32$. The errors and runtimes for each decomposition are shown in \Cref{tab:ex4_error_V1}, \Cref{tab:ex4_error_V2}, and \Cref{tab:ex4_runtime_V1}.

We find that the multi-step FGGC can further improve efficiency for certain decompositions. The $1000\times8$ decomposition achieves the best performance.  As the final time is divided into more segments (i.e., as $T_\text{multi}$ increases), the total CPU runtime first decreases and then increases. This is because the number of ODE evolution steps may decrease, but the number of initial decompositions and reconstructions increase with $T_\text{multi}$. Regarding error, as the number of segments increases, the accumulated iteration error also grows. The results in \Cref{tab:ex4_error_V1} and \Cref{tab:ex4_error_V2} show that the iteration error accumulates by a factor of several times, but remains acceptable in our experiment.

Next, we consider a two-dimensional example: 
\begin{example}
    Take 
    \begin{equation}\begin{aligned}
        u_0(x,y)=\sqrt{\frac{64}{\pi}}\operatorname{exp}\left(-32(x^2+y^2)+\frac{\mathrm{i}}{\varepsilon}(x+y)\right),
        \label{eq:ex51}
    \end{aligned}\end{equation}
    and  
    \begin{equation}\begin{aligned}
        V_1(x,y)=1-e^{x^2+y^2} \text{\ \ or\ \ }V_2(x,y)=\frac{x^2+y^2}{2}.
        \label{eq:ex52}
    \end{aligned}\end{equation}
    We solve on the spatial interval $[-2,2]^2$ with $\varepsilon=1/2^6$ and final time $T=1.6s$, using the multi-step version FGGC.
    \label{ex:ex5}
\end{example}

We use the same mesh steps as in \Cref{section:4.2} and the neighbor strategy \textbf{Q2P2}. The $16000$ evolution steps are decomposed into $8000\times2$, $4000\times4$, $2000\times8$, $1000\times16$, and $500\times32$. The errors are shown in \Cref{tab:ex5_error_V1} and \Cref{tab:ex5_error_V2}, demonstrating that the multi-step algorithm does not introduce significant iteration error. The runtimes are presented in \Cref{tab:ex5_runtime_V1}, where the $2000\times8$ decomposition achieves the best efficiency. 


\begin{table}[p]
\centering
\begin{tabular}{ll *{5}{c}}
\toprule
\multicolumn{2}{c}{\textbf{$\varepsilon$}} & $1/2^{6}$ & $1/2^{8}$ & $1/2^{10}$ & $1/2^{12}$ & $1/2^{14}$ \\
\midrule
\multicolumn{2}{c}{\textbf{FGA}} & 4.097E-03 & 5.808E-04 & 1.349E-04 & 5.478E-05 & 5.471E-05 \\
\midrule
\multirow{3}{*}{\textbf{FGGC}} & \textbf{Q2P2} & 4.118E-03 & 6.011E-04 & 1.949E-04 & 1.475E-04 & 1.379E-04 \\
& \textbf{Q4P2} & 4.097E-03 & 5.809E-04 & 1.350E-04 & 5.486E-05 & 5.474E-05 \\
& \textbf{Q4P4} & 4.097E-03 & 5.808E-04 & 1.349E-04 & 5.477E-05 & 5.470E-05 \\
\bottomrule
\end{tabular}
\caption{Example 1, the overall error of FGA and FGGC (in comparison with TSSP).}
\label{tab:ex1_totalerror_V1}

\vspace{1cm}
\centering
\begin{tabular}{ll *{5}{c}}
\toprule
\multicolumn{2}{c}{\textbf{$\varepsilon$}} & $1/2^{6}$ & $1/2^{8}$ & $1/2^{10}$ & $1/2^{12}$ & $1/2^{14}$\\
\midrule
\multirow{3}{*}{\textbf{FGGC}} & \textbf{Q2P2} & 1.135E-04 & 1.186E-04 & 1.279E-04 & 1.319E-04 & 1.300E-04\\
& \textbf{Q4P2} & 9.788E-07 & 8.641E-07 & 9.525E-07 & 8.141E-07 & 8.195E-07\\
& \textbf{Q4P4} & 9.295E-08 & 8.712E-08 & 8.421E-08 & 7.017E-08 & 7.902E-08\\
\bottomrule
\end{tabular}
\caption{Example 1, the LSA error of FGGC (in comparison with FGA).}
\label{tab:ex1_lsaerror_V1}

\vspace{1cm}
\centering
\begin{tabular}{cc *{5}{c}}
\toprule
\multicolumn{2}{c}{\textbf{$\varepsilon$}} & $1/2^{6}$ & $1/2^{8}$ & $1/2^{10}$ & $1/2^{12}$ & $1/2^{14}$\\
\midrule
\multicolumn{2}{c}{\textbf{TSSP}} & 0.009 & 0.040 & 0.187 & 0.846 & 4.114\\
\midrule
\multicolumn{2}{c}{\textbf{FGA}} & 0.125 & 0.156 & 0.269 & 0.516 & 1.055\\
\midrule
\multirow{3}{*}{\textbf{FGGC}} & \textbf{Q2P2} & 0.125 & 0.155 & 0.264 & 0.504 & 0.993\\
& \textbf{Q4P2} & 0.125 & 0.155 & 0.263 & 0.501 & 0.994\\
& \textbf{Q4P4} & 0.129 & 0.155 & 0.263 & 0.501 & 0.991\\
\bottomrule
\end{tabular}
\caption{Example 1, the CPU runtime of TSSP, FGA and FGGC in sequential scenario.}
\label{tab:ex1_sequtime}

\vspace{1cm}
\centering
\begin{tabular}{cc *{5}{c}}
\toprule
\multicolumn{2}{c}{\textbf{$\varepsilon$}} & $1/2^{6}$ & $1/2^{8}$ & $1/2^{10}$ & $1/2^{12}$ & $1/2^{14}$\\
\midrule
\multicolumn{2}{c}{\textbf{TSSP}} & 0.010 & 0.043 & 0.187 & 0.846 & 3.147\\
\midrule
\multicolumn{2}{c}{\textbf{FGA}} & 0.013 & 0.013 & 0.030 & 0.060 & 0.176\\
\midrule
\multirow{3}{*}{\textbf{FGGC}} & \textbf{Q2P2} & 0.009 & 0.009 & 0.015 & 0.037 & 0.073\\
& \textbf{Q4P2} & 0.009 & 0.009 & 0.010 & 0.029 & 0.066\\
& \textbf{Q4P4} & 0.009 & 0.009 & 0.010 & 0.032 & 0.072\\
\bottomrule
\end{tabular}
\caption{Example 1, the CPU runtime of TSSP, FGA and FGGC in parallel scenario.}
\label{tab:ex1_paratime}
\end{table}

\begin{table}[p]
\centering
\begin{tabular}{ll *{5}{c}}
\toprule
\multicolumn{2}{c}{\textbf{$\varepsilon$}} & $1/2^{6}$ & $1/2^{8}$ & $1/2^{10}$ & $1/2^{12}$ & $1/2^{14}$ \\
\midrule
\multicolumn{2}{c}{\textbf{FGA}} & 5.453E-03 & 1.412E-03 & 3.620E-04 & 1.419E-04 & - \\
\midrule
\multirow{3}{*}{\textbf{FGGC}} & \textbf{Q2P2} & 5.460E-03 & 1.414E-03 & 3.688E-04 & 1.472E-04 & - \\
&                                \textbf{Q4P2} & 5.453E-03 & 1.412E-03 & 3.620E-04 & 1.419E-04 & - \\
&                                \textbf{Q4P4} & 5.453E-03 & 1.412E-03 & 3.621E-04 & 1.419E-04 & - \\
\bottomrule
\end{tabular}
\caption{Example 2, the overall error of FGA and FGGC (in comparison with TSSP) with potential $V_1$.}
\label{tab:ex2_totalerror_V1}

\vspace{1cm}
\centering
\begin{tabular}{ll *{5}{c}}
\toprule
\multicolumn{2}{c}{\textbf{$\varepsilon$}} & $1/2^{6}$ & $1/2^{8}$ & $1/2^{10}$ & $1/2^{12}$ & $1/2^{14}$\\
\midrule
\multirow{3}{*}{\textbf{FGGC}} & \textbf{Q2P2} & 2.151E-05 & 2.173E-05 & 2.069E-05 & 1.950E-05 &  1.884E-05\\
&                                \textbf{Q4P2} & 9.007E-08 & 8.509E-08 & 8.482E-08 & 8.628E-08 &  8.434E-08\\
&                                \textbf{Q4P4} & 5.680E-08 & 6.012E-08 & 5.818E-08 & 5.850E-08 &  5.824E-08\\
\bottomrule
\end{tabular}
\caption{Example 2, the LSA error of FGGC (in comparison with FGA) with potential $V_1$.}
\label{tab:ex2_LSAerror_V1}

\vspace{1cm}
\centering
\begin{tabular}{ll *{5}{c}}
\toprule
\multicolumn{2}{c}{\textbf{$\varepsilon$}} & $1/2^{6}$ & $1/2^{8}$ & $1/2^{10}$ & $1/2^{12}$ & $1/2^{14}$ \\
\midrule
\multicolumn{2}{c}{\textbf{FGA}} & 6.640E-06 & 1.445E-05 & 6.659E-05 & 1.154E-04 & -\\
\midrule
\multirow{3}{*}{\textbf{FGGC}} & \textbf{Q2P2} & 1.698E-05 & 3.268E-05 & 6.962E-05 & 1.167E-04 & - \\
&                                \textbf{Q4P2} & 6.645E-06 & 1.446E-05 & 6.659E-05 & 1.154E-04 & -\\
&                                \textbf{Q4P4} & 6.640E-06 & 1.445E-05 & 6.659E-05 & 1.153E-04 & -\\
\bottomrule
\end{tabular}
\caption{Example 2, the overall error of FGA and FGGC (in comparison with TSSP) with potential $V_2$.}
\label{tab:ex2_totalerror_V2}

\vspace{1cm}
\centering
\begin{tabular}{ll *{5}{c}}
\toprule
\multicolumn{2}{c}{\textbf{$\varepsilon$}} & $1/2^{6}$ & $1/2^{8}$ & $1/2^{10}$ & $1/2^{12}$ & $1/2^{14}$\\
\midrule
\multirow{3}{*}{\textbf{FGGC}} & \textbf{Q2P2} & 1.600E-05 & 2.942E-05 & 2.034E-05 & 1.753E-05 & 1.832E-05\\
&                                \textbf{Q4P2} & 1.231E-07 & 1.378E-07 & 1.254E-07 & 1.347E-07 & 1.329E-07\\
&                                \textbf{Q4P4} & 6.201E-08 & 6.288E-08 & 6.100E-08 & 6.068E-08 & 6.331E-08\\
\bottomrule
\end{tabular}
\caption{Example 2, the LSA error of FGGC (in comparison with FGA) with potential $V_2$.}
\label{tab:ex2_LSAerror_V2}

\vspace{1cm}
\centering
\begin{tabular}{cc *{5}{c}}
\toprule
\multicolumn{2}{c}{\textbf{$\varepsilon$}} & $1/2^{6}$ & $1/2^{8}$ & $1/2^{10}$ & $1/2^{12}$ & $1/2^{14}$\\
\midrule
\multicolumn{2}{c}{\textbf{TSSP}} & 3.676 & 19.756 & 384.818 & 4962.956 & - \\
\midrule
\multicolumn{2}{c}{\textbf{FGA}} & 2.278 & 5.468 & 37.638 & 447.061 & 6746.639 \\
\midrule
\multirow{3}{*}{\textbf{FGGC}} & \textbf{Q2P2} & 2.022 & 3.281 & 11.115 & 54.325 & 443.103 \\
& \textbf{Q4P2} & 2.066 & 3.397 & 11.355 & 56.077 & 459.125 \\
& \textbf{Q4P4} &2.368 & 3.784 & 12.382 & 60.031 & 473.794 \\
\bottomrule
\end{tabular}
\caption{Example 2, the CPU runtime of TSSP, FGA and FGGC.}
\label{tab:ex2_totaltime}
\end{table}

\begin{table}[p]
\centering
\begin{tabular}{ll *{4}{c}}
\toprule
\multicolumn{2}{c}{\textbf{$\varepsilon$}} & $1/2^{5}$ & $1/2^{6}$ & $1/2^{7}$ & $1/2^{8}$ \\
\midrule
\multicolumn{2}{c}{\textbf{FGA}} & 2.933E-03 & 1.653E-03 & 2.903E-03 & 5.997E-03 \\
\midrule
\multirow{2}{*}{\textbf{FGGC}} & \textbf{Q2P2} & 2.938E-03 & 1.645E-03 & 2.904E-03 & 5.998E-03 \\
&                                \textbf{Q4P2} & 2.933E-03 & 1.653E-03 & 2.903E-03 & 5.997E-03 \\
\bottomrule
\end{tabular}
\caption{Example 3, the overall error for FGA and FGGC (in comparison with TSSP).}
\label{tab:ex3_totalerror}

\vspace{1cm}
\centering
\begin{tabular}{ll *{4}{c}}
\toprule
\multicolumn{2}{c}{\textbf{$\varepsilon$}} & $1/2^{5}$ & $1/2^{6}$ & $1/2^{7}$ & $1/2^{8}$\\
\midrule
\multirow{2}{*}{\textbf{FGGC}} & \textbf{Q2P2} & 5.002E-05 & 2.050E-05 & 5.383E-05 & 3.716E-05 \\
&                                \textbf{Q4P2} & 5.899E-07 & 2.326E-07 & 5.365E-07 & 2.284E-07 \\
\bottomrule
\end{tabular}
\caption{Example 3, the LSA error for FGGC (in comparison with FGA).}
\label{tab:ex3_LSAerror}

\vspace{1cm}
\centering
\begin{tabular}{cc *{4}{c}}
\toprule
\multicolumn{2}{c}{\textbf{$\varepsilon$}} & $1/2^{5}$ & $1/2^{6}$ & $1/2^{7}$ & $1/2^{8}$\\
\midrule
\multicolumn{2}{c}{\textbf{TSSP}} & 42.575 & 314.893 & 2018.370 & 17884.954 \\
\midrule
\multicolumn{2}{c}{\textbf{FGA}} & 1179.614 & 2346.721 & 6954.061 & 24422.830 \\
\midrule
\multirow{2}{*}{\textbf{FGGC}} & \textbf{Q2P2} & 364.995 & 343.005 & 476.449 & 1070.442 \\
& \textbf{Q4P2} & 417.945 & 396.375 & 591.764 & 1307.381 \\
\bottomrule
\end{tabular}
\caption{Example 3, the CPU runtime for TSSP, FGA and FGGC.}
\label{tab:ex3_totaltime}
\end{table}

\begin{table}[p]
\centering
\small
\begin{tabular}{l *{7}{c}}
\toprule
\multirow{2}{*}{\textbf{\ \ FGA}} & \multicolumn{6}{c}{\textbf{FGGC}} \\
\cmidrule(lr){2-7}
             & \textbf{$8000\times1$} & \textbf{$4000\times2$} & \textbf{$2000\times4$} & \textbf{$1000\times8$} & \textbf{$500\times16$} & \textbf{$250\times32$}\\
\midrule
5.478E-05 & 1.475E-04 & 1.406E-04 & 3.073E-04 & 3.628E-04 & 4.217E-04 & 4.250E-04 \\
\bottomrule
\end{tabular}
\caption{Example 4, the overall error for FGA and multi-step FGGC with potential $V_1$.}
\label{tab:ex4_error_V1}

\vspace{1cm}
\centering
\small
\begin{tabular}{l *{7}{c}}
\toprule
\multirow{2}{*}{\textbf{\ \ FGA}} & \multicolumn{6}{c}{\textbf{FGGC}} \\
\cmidrule(lr){2-7}
             & \textbf{$8000\times1$} & \textbf{$4000\times2$} & \textbf{$2000\times4$} & \textbf{$1000\times8$} & \textbf{$500\times16$} & \textbf{$250\times32$}\\
\midrule
3.151E-05 & 3.376E-05 & 6.003E-05 & 1.147E-04 & 1.582E-04 & 1.755E-04 & 2.509E-04 \\
\bottomrule
\end{tabular}
\caption{Example 4, the overall error for FGA and multi-step FGGC with potential $V_2$.}
\label{tab:ex4_error_V2}

\vspace{1cm}
\centering
\small
\begin{tabular}{l *{8}{c}}
\toprule
\multirow{2}{*}{\textbf{TSSP}} & \multirow{2}{*}{\textbf{FGA}} & \multicolumn{6}{c}{\textbf{FGGC}} \\
\cmidrule(lr){3-8}
 &  & \textbf{$8000\times1$} & \textbf{$4000\times2$} & \textbf{$2000\times4$} & \textbf{$1000\times8$} & \textbf{$500\times16$} & \textbf{$250\times32$}\\
\midrule
1.132 & 0.058 & 0.041 & 0.035 & 0.024 & 0.021 & 0.030 & 0.064 \\
\bottomrule
\end{tabular}
\caption{Example 4, the CPU runtime of TSSP, FGA and multi-step FGGC.}
\label{tab:ex4_runtime_V1}

\vspace{1cm}
\centering
\small
\begin{tabular}{l *{7}{c}}
\toprule
\multirow{2}{*}{\textbf{\ \ FGA}} & \multicolumn{6}{c}{\textbf{FGGC}} \\
\cmidrule(lr){2-7}
             & \textbf{$16000\times1$} & \textbf{$8000\times2$} & \textbf{$4000\times4$} & \textbf{$2000\times8$} & \textbf{$1000\times16$} & \textbf{$500\times32$}\\
\midrule
4.773E-03 & 4.723E-03 & 3.974E-03 & 2.807E-03 & 2.801E-03 & 2.884E-03 & 2.995E-03 \\
\bottomrule
\end{tabular}
\caption{Example 5, the overall error for FGA and multi-step FGGC with potential $V_1$.}
\label{tab:ex5_error_V1}

\vspace{1cm}
\centering
\small
\begin{tabular}{l *{7}{c}}
\toprule
\multirow{2}{*}{\textbf{\ \ FGA}} & \multicolumn{6}{c}{\textbf{FGGC}} \\
\cmidrule(lr){2-7}
             & \textbf{$16000\times1$} & \textbf{$8000\times2$} & \textbf{$4000\times4$} & \textbf{$2000\times8$} & \textbf{$1000\times16$} & \textbf{$500\times32$}\\
\midrule
6.640E-06 & 7.793E-05 & 1.947E-05 & 6.855E-05 & 8.337E-05 & 1.936E-04 & 3.130E-04 \\
\bottomrule
\end{tabular}
\caption{Example 5, the overall error for FGA and multi-step FGGC with potential $V_2$.}
\label{tab:ex5_error_V2}

\vspace{1cm}
\centering
\small
\begin{tabular}{l *{8}{c}}
\toprule
\multirow{2}{*}{\textbf{TSSP}} & \multirow{2}{*}{\textbf{FGA}} & \multicolumn{6}{c}{\textbf{FGGC}} \\
\cmidrule(lr){3-8}
 &  & \textbf{$16000\times1$} & \textbf{$8000\times2$} & \textbf{$4000\times4$} & \textbf{$2000\times8$} & \textbf{$1000\times16$} & \textbf{$500\times32$}\\
\midrule
7.510 & 4.329 & 4.120 & 4.083 & 3.173 & 2.305 & 2.326 & 3.053 \\
\bottomrule
\end{tabular}
\caption{Example 5, the CPU runtime of TSSP, FGA and multi-step FGGC.}
\label{tab:ex5_runtime_V1}
\end{table}

\section{Conclusion and Discussion}\label{sec:5}

In this paper, we have introduced the frozen Gaussian grid-point correction (FGGC) method for efficiently computing solutions to the linear Schr{\"o}dinger equation in the semi-classical regime. Building upon the frozen Gaussian approximation, FGGC uses a straightforward ``on-grid correction'' strategy based on the least squares approximation. This approach demonstrates significant potential for reducing computational costs during the wave reconstruction step by decomposing off-grid Gaussian wave-packets into on-grid counterparts, thereby enabling the use of highly efficient fast Fourier transform algorithm. We have theoretically proved that the error of the least squares approximation has a uniform upper bound independent of $\varepsilon$. Numerical experiments confirm that this error is indeed independent of $\varepsilon$ and decreases rapidly as more neighbors are included. The FGGC algorithm achieves substantial efficiency improvements while introducing only negligible additional error. Furthermore, a multi-step version of the FGGC algorithm is proposed, which further enhances computational efficiency. The FGGC framework is not limited to the linear Schr\"odinger equation; it can be extended to other multi-scale problems suitable for the frozen Gaussian approximation, such as high-frequency wave propagation and hyperbolic systems.

\end{document}